\documentclass[11pt,oneside,english]{amsart}
\usepackage[T1]{fontenc}
\usepackage[latin9]{inputenc}
\usepackage{color}
\usepackage{babel}
\usepackage{amstext}
\usepackage{amsthm}
\usepackage{amssymb}
\usepackage{graphicx}
\usepackage[unicode=true,pdfusetitle,
 bookmarks=true,bookmarksnumbered=false,bookmarksopen=false,
 breaklinks=false,pdfborder={0 0 0},backref=false,colorlinks=true]
 {hyperref}
\hypersetup{
 citecolor=blue}

\makeatletter
\numberwithin{equation}{section}
\numberwithin{figure}{section}
\@ifundefined{lettrine}{\usepackage{lettrine}}{}
  \theoremstyle{remark}
  \newtheorem*{acknowledgement*}{\protect\acknowledgementname}
\theoremstyle{plain}
\newtheorem{thm}{\protect\theoremname}[section]
  \theoremstyle{definition}
  \newtheorem{defn}[thm]{\protect\definitionname}
  \theoremstyle{remark}
  \newtheorem{rem}[thm]{\protect\remarkname}
  \theoremstyle{plain}
  \newtheorem{prop}[thm]{\protect\propositionname}
  \theoremstyle{plain}
  \newtheorem{lem}[thm]{\protect\lemmaname}
  \theoremstyle{plain}
  \newtheorem{cor}[thm]{\protect\corollaryname}
  \theoremstyle{plain}
  \newtheorem*{thm*}{\protect\theoremname}

\usepackage{geometry}
\usepackage{amsmath}
\usepackage{amsthm}

\numberwithin{equation}{section}
\numberwithin{figure}{section}
\usepackage{enumitem}		
 \let\footnote=\endnote
\@ifundefined{lettrine}{\usepackage{lettrine}}{}

\theoremstyle{definition}
\newtheorem{thmx}{Theorem}

\def\d{\delta}

\def\k{\kappa}

\def\G{\Gamma}

\def\g{\gamma}

\def\w{\omega}
\def\dd{\mathrm{d}}

\def\D{\mathbb{D}}
\def\R{\mathbb{R}}

\def\N{\mathbb{N}}

\def\ep{\varepsilon}
\def\vphi{\varphi}

\def\cal{\mathcal}

\def\F{\mathcal{F}}

\def\lt{\lambda_{T}}

\def\gt{G_T}

\usepackage{float}

\address{Department of Mathematics,
The Ohio State University, Columbus, OH 43210, USA} 
\email{chen.8022@osu.edu}
\address{Department of Mathematics, The University of Chicago, Chicago, IL 60637, USA} 
\email{lkao@math.uchicago.edu}
\address{Department of Mathematics, The University of Chicago, Chicago, IL 60637, USA} 
\email{kihopark@math.uchicago.edu}


\subjclass[2000]{}
\thanks{This material is based upon work supported by the National Science Foundation (NSF) under Grant Number DMS 1641020. The second author gratefully acknowledges support from the NSF Postdoctoral Research Fellowship under Grant Number DMS 1703554.}

\def\G{\Gamma}
\def\g{\gamma}

\def\D{\Delta}
\def\R{\mathbb{R}}

\def\ep{\varepsilon}

\def\F{\mathcal{F}}

\def\lt{\lambda_{T}}

\def\gt{\cal{G}_T}

\def\Sing{\mathrm{Sing}}
\def\Reg{\mathrm{Reg}}
\def\vp{\varphi}

\makeatother

  \providecommand{\acknowledgementname}{Acknowledgement}
  \providecommand{\corollaryname}{Corollary}
  \providecommand{\definitionname}{Definition}
  \providecommand{\lemmaname}{Lemma}
  \providecommand{\propositionname}{Proposition}
  \providecommand{\remarkname}{Remark}
  \providecommand{\theoremname}{Theorem}
\providecommand{\theoremname}{Theorem}

\begin{document}

\title[Equilibriums for Geodesic Flows over Surfaces without Focal Points]{Unique Equilibrium States for Geodesic Flows over Surfaces without
Focal Points}

\author{Dong Chen, Lien-Yung Kao, and Kiho Park}

\date{\today}

\maketitle

\begin{abstract}
In this paper, we study dynamics of geodesic flows over closed surfaces
of genus greater than or equal to 2 without focal points. Especially,
we prove that there is a large class of potentials having unique equilibrium
states, including scalar multiples of the geometric potential, provided
the scalar is less than 1. Moreover, we discuss ergodic properties
of these unique equilibrium states. We show these unique equilibrium
states are Bernoulli, and weighted regular periodic orbits are equidistributed
relative to these unique equilibrium states.
\end{abstract}

{   \hypersetup{linkcolor=black}   \tableofcontents }

\begingroup 
\color{black}

\endgroup   

\section{Introduction}

This paper is devoted to the study of dynamics of the geodesic flows
over closed surfaces without focal points. We focus on the thermodynamic
formalism of the geodesic flows, especially, the uniqueness of the
equilibrium states and their ergodic properties. For uniformly hyperbolic
flows, also known as Anosov flows, thanks to fundamental works of
Ruelle, Bowen, and Ratner, we know that every H\"{o}lder potential
has a unique equilibrium state which enjoys several ergodic features
such as Bernoulli and equidistribution properties. It is also well-known
that the geodesic flow on a negatively curved manifold is uniformly
hyperbolic. However, when the manifold contains subsets with zero
or positive curvature, the geodesic flow becomes non-uniformly hyperbolic.
The non-uniform hyperbolicity greatly increases the difficulty in
understanding the thermodynamics of these flows. Nevertheless, the
geometric features of surfaces without focal points allows us to investigate
the dynamics of the geodesic flows. There are several geometric properties
are available in this setting such as the flat strip theorem, $C^{2}$-regularity
of the horocycles, and more. These properties enable us to derive
and extend the existence of unique of measure of maximal entropy \cite{Knieper:1998ht}
and the equilibrium state \cite{BCFT2017} on closed rank 1 nonpositively
curved manifolds to closed surfaces without focal points and genus
at least 2.

Combining the dynamical and geometric features of surfaces without
focal points, in this paper, we are able to prove the uniqueness of
equilibrium states for a large class of potentials and Bernoulli and
equidistribution properties for such equilibrium states. These results
also generlize Gelfert-Ruggiero's recent work \cite{Gelfert:2017tx}
on the uniqueness of measure of the maximum entropy for the geodesics
flows over surfaces without focal points.

Putting our results in contexts below, we shall first introduce relevant
terminologies briefly (please see Section \ref{sec:Preliminaries-of-dynamics}
and \ref{sec:Preliminaries-of-geometry} for more details). Throughout
the paper, $S$ denotes a closed (i.e., compact without boundary)
$C^{\infty}$ Riemannian surface of genus greater than or equal to
2 without focal points. We denote the geodesics flow on the unit tangent
bundle $T^{1}S$ by ${\cal F}=(f_{t})_{t\in\R}$ .

The thermodynamic objects that we are interested in this paper are
topological pressure and equilibrium states. For a continuous potential
(i.e., function) $\vp:T^{1}S\to\R$, the \textit{topological pressure
}$P(\vp)$ of $\vp$ with respect to ${\cal F}$ can be described
by the \textit{variational principle}: 
\[
P(\vp)=\sup\{h_{\mu}({\cal F})+\int\vp d\mu:\ \mu{\rm \ is\ }a\mbox{\ }{\cal F}-{\rm invariant\ }{\rm Borel\ }{\rm probability\ }{\rm measure}\}
\]
where $h_{\mu}({\cal F})$ is the measure-theoretic entropy of $\mu$
with respect to ${\cal F}$. An invariant Borel probability measure
$\mu$ achieving the supremum is called an \textit{equilibrium state}.
We notice that when $\vp$ is identically equal to $0$ then $P(0)$
is equal to the topological entropy $h_{{\rm top}}({\cal F})$ of
${\cal F}$, and an equilibrium state for $\vp\equiv0$ is called
a \textit{measure of maximum entropy}.

The non-uniform hyperbolicity of ${\cal F}$ comes from the existence
of the \textit{singular set} $\Sing$. For surfaces without focal
points, we can describe the singular set as $\Sing=\{v\in T^{1}S:\ K(\pi f_{t}v)=0\ \forall t\in\R\}$
where $\pi:T^{1}S\to S$ is the canonical projection and $K$ is the
Gaussian curvature (see Section \ref{sec:Preliminaries-of-geometry}
for alternative characterizations of the singular set). The complement
of $\Sing$ is called the \textit{regular set} and denoted by $\Reg$.

Our first result asserts the uniqueness of the equilibrium states
for the potentials with \textquotedblleft nice\textquotedblright{}
regularity and carrying smaller pressure on the singular set. The
potentials with \textquotedblleft nice\textquotedblright{} regularity
include H\"{o}lder potentials and the \textit{geometric potential}
$\vp^{u}$ defined as 
\[
\varphi^{u}(v):=-\lim_{t\to0}\frac{1}{t}{\displaystyle \log\det(}\left.df_{t}\right|_{E^{u}(v)}).
\]

\begin{thmx}

Let $S$ be a surface of genus greater than or equal to 2 without
focal points and ${\cal F}$ be the geodesic flow over $S$. Let $\vp:T^{1}S\to\R$
be a H\"{o}lder continuous potential or a scalar multiple of the
geometrical potential. Suppose $\vp$ verifies the pressure gap property
$P({\rm Sing},\vphi)<P(\vphi)$, then $\vphi$ has a unique equilibrium
state $\mu_{\vphi}$.

\end{thmx}

The proof of Theorem A uses the same idea as the proof of \cite[Theorem A]{BCFT2017}.
Both \cite{BCFT2017} and this paper follow the general framework
introduced by Bowen \cite{Bowen:1974df}, which was subsequently extended
by Franco \cite{Franco:1977jy} and recently extended further by Climenhaga
and Thompson \cite{Climenhaga:2016ut}. We have more detailed discussion
of this method in Section \ref{sec:Preliminaries-of-dynamics}. Roughly
speaking, the general framework follows the original work of Bowen
stating that when the potential has ``nice'' regularity (namely,
the Bowen property) and the system has ``sufficient hyperbolicity''
(namely, the specification property) then this potential has a unique
equilibrium state. We discuss more details of this method in Section
\ref{sec:Preliminaries-of-dynamics} and Section \ref{sec:Preliminaries-of-geometry}.

The second result, following Theorem A, we discuss several ergodic
properties of these unique equilibrium states. We successfully extend
several properties known to hold under uniform hyperbolic cases (cf.,
for example, \cite{Parry:1990tn}), as well as under nonpositively
curved surfaces (cf., for example, \cite{Pollicott:1996jq}, \cite{Ledrappier:2016cn}
and \cite{BCFT2017}). Namely, these unique equilibrium states are
Bernoulli and the weak{*} limit of the weighted regular periodic orbits.
Recall that other weaker ergodic properties such as being Kolmogorov
and strongly mixing follows once the measure is Bernoulli.

\begin{thmx}

Suppose $\vp$ satisfies the same assumptions in Theorem A. Then the
unique equilibrium state $\mu_{\vphi}$ is fully supported, $\mu_{\varphi}({\rm Reg})=1$,
Bernoulli, and a weak$^{*}$ limit of the weighted regular periodic
orbits.

\end{thmx}

In our last main result, we study on the geometric potential $\vp^{u}$
and its pressure function $q\mapsto P(q\vp^{u})$. We give the full
description of the pressure function, and show that the situation
is analogous to the nonpositively curved manifolds (see, for example,
\cite{Burns:2014gq} and \cite{BCFT2017}).

\begin{thmx}

Under the same assumption in Theorem A, suppose $\varphi=q\varphi^{u}$
is a scalar multiple of the geometric potential. Then, for $q<1$,
$\vphi$ has a unique equilibrium state $\mu_{q}$ which is, fully
supported, $\mu_{\varphi}({\rm Reg})=1$, Bernoulli, and the weak$^{*}$
limit of the weighted regular periodic orbits. Furthermore, the map
$q\mapsto P(q\vphi^{u})$ is $C^{1}$ for $q<1$, and $P(q\vp^{u})=0$
for $q\geq1$ when $\Sing\neq\emptyset$. 

\begin{figure}[H]

\begin{center}\includegraphics[scale=0.5]{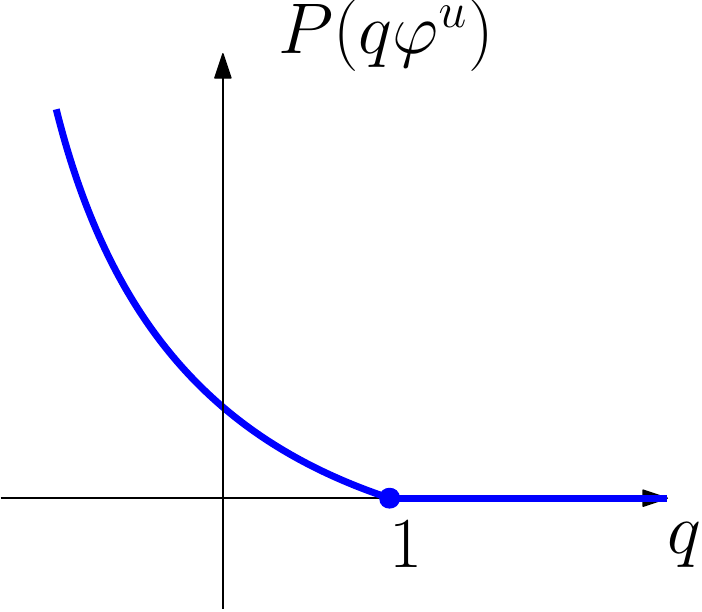}\end{center}

\caption{Pressure function}

\end{figure}

\end{thmx}

This paper is organized as follows. In Section \ref{sec:Preliminaries-of-dynamics},
we go over the background in thermodynamic formalism; in particular,
we introduce our primary tool, the Climenhaga-Thompson program. In
Section \ref{sec:Preliminaries-of-geometry}, we recall the definitions
and geometric features of surfaces and manifolds without focal points.
Section \ref{sec:Decompositions-for-finite}, \ref{sec:The-specification-property},
and \ref{sec:The-Bowen-property} are devoted to setting up the framework
for the Climenhaga-Thompson program, namely, orbit decomposition,
the specification property, and the Bowen property. We will prove
Theorem A in Section \ref{sec:Pressure-gap-and_Thm_A} and Theorem
B in Section \ref{sec:Prop_eq_states_Thm_B}. In Section \ref{sec:Thm_C_exampels},
we will show Theorem C and provide some examples of potentials satisfying
Theorem A.

\begin{acknowledgement*}
The authors are incredibly grateful to Vaughn Climenhaga and Daniel
Thompson for not only organizing the 2017 MRC program in dynamical
systems and proposing this project to the authors, but also for their
continuous support and help after the MRC program. This paper would
not have been possible without them. The authors would like to thank
Matthew Smith and R\'{e}gis Var\~{a}o for initiating and discussing
this MRC project with them, Keith Burns for sharing many insightful
ideas, and to Amie Wilkinson for suggestions on improving the presentation
of this paper.
\end{acknowledgement*}

\section{Preliminaries of dynamics\label{sec:Preliminaries-of-dynamics}}

In this section, we introduce necessary background in thermodynamics.
An excellent reference for terminologies introduced in this section
is Walters' book \cite{Walters:2000vc}.

Throughout this section, $(X,d)$ is a compact metric space, ${\cal F}=(f_{t})_{t\in\R}$
is a continuous flow on $X$, and $\vp:X\to\R$ is a continuous potential.

\subsection{Topological pressure}

For the convenience, we first define the following terms.

\begin{defn}
For any $t,\delta>0$ and $x\in X$,

\begin{enumerate}[font=\normalfont] 

\item The \textit{Bowen ball} of radius $\delta$ and order $t$
at $x$ is defined as 
\[
B_{t}(x,\delta)=\{y\in X:\ d(f_{\tau}x,f_{\tau}y)<\delta\ {\rm for}\ {\rm all\ }0\leq t\leq\tau\}.
\]

\item We say a set $E$ is $(t,\delta)-$\textit{separated} if for
all $x,y\in E$ with $x\neq y$, there exists $t_{0}\in[0,t]$ such
that $d(f_{t_{0}}x,f_{t_{0}}y)\geq\delta$. 

\end{enumerate}
\end{defn}

\begin{defn}
[Finite length orbit segments] Any subset ${\cal C}\subset X\times[0,\infty)$
can be identified with a collection of \textit{finite length orbit
segments}. More precisely, 
\[
{\cal C}=\{(x,t)\in X\times[0,\infty)\}
\]
 where $(x,t)$ is identified with the orbit $\{f_{\tau}x:\ 0\leq\tau\leq t\}.$
We denote $\Phi(x,t):=\int_{0}^{t}\vp(f_{\tau}x)d\tau$ the integral
of $\vp$ along an orbit segment $(x,t)$.
\end{defn}

Let ${\cal C}_{t}:=\{x\in X:\ (x,t)\in{\cal C}\}$ be the set of length
$t$ orbit segments in ${\cal C}$. We define 
\[
\Lambda({\cal C},\vp,\delta,t)=\sup\{\sum_{x\in E}e^{\Phi(x,t)}:\ E\subset{\cal C}_{t}\ {\rm is}\ (t,\delta)-{\rm separated}\}.
\]

\begin{defn}
[Topological pressure] The \textit{pressure} of $\vp$ on ${\cal C}$
is defined as 
\[
P({\cal C},\vp)=\lim_{\delta\to0}\limsup_{t\to\infty}\frac{1}{t}\log\Lambda({\cal C},\vp,\delta,t).
\]
When ${\cal C=}X\times[0,\infty)$, we denote $P(X\times[0,\infty),\vp)$
by $P(\vp)$ and call it the \textit{topological pressure} of $\vp$
with respect to ${\cal F}$.
\end{defn}

As noted in the introduction, the pressure $P(\vp)$ satisfies the
variational principle
\[
P(\vp)=\sup_{\mu\in{\cal M}({\cal F})}\{h_{\mu}({\cal F})+\int\vp d\mu\}
\]
where ${\cal M}({\cal F})$ is the set of ${\cal F}-$invariant probability
measures on $X$. Also, a ${\cal F}-$invariant probability measure
$\mu$ realizing the supremum is called an \textit{equilibrium state}
for $\vp$.

\begin{rem}
$\ $\begin{enumerate}[font=\normalfont] 

\item When the entropy map $\mu\mapsto h_{\mu}$ is upper semi-continuous,
any weak{*} limit of a sequence invariant measures approximating the
pressure is an equilibrium state. In particular, there exists at least
one equilibrium state for every continuous potentials.

\item In our setting, the geodesics flow over surfaces without focal
points, the upper semi-continuity of the entropy map is guaranteed
by the entropy-expansivity established in \cite{Liu:2016de}.

\end{enumerate}
\end{rem}

\subsection{Climenhaga-Thompson's criteria for the uniqueness of equilibrium
states}

Climenhaga and Thompson have a series of successful results on establishing
the uniqueness of the equilibrium states of various non-uniformly
hyperbolic systems; see \cite{Climenhaga:2012ed,Climenhaga:2013gm,Climenhaga:2015wf,Climenhaga:2016ut,BCFT2017}.
This work follows the same method, so called, the Climenhaga-Thompson
program. In this subsection, we aim to introduce terms using in the
Climenhaga-Thompson program.

One of the primary ideas in the Climenhaga-Thompson program is to
relax the original assumptions from the work of Bowen on the uniformly
hyperbolic systems \cite{Bowen:1974df} by asking that the \textquotedblleft hyperbolic\textquotedblright{}
behavior on the system and the \textquotedblleft good regularity\textquotedblright{}
on the potential to hold on a (large) collection of finite orbit segments
${\cal C}$ rather than in the whole space. This flexibility is essential
for applying this method to non-uniformly hyperbolic systems. To be
more precise, the ``hyperbolic'' behavior refers to the \textit{specification}
and the ``good regularity'' refers to the potential having the\textit{
Bowen property}.

\begin{defn}
[Specification] We say ${\cal C}\subset X\times[0,\infty)$ has \textit{specification}
at scale $\rho>0$ if there exists $\tau=\tau(\rho)$ such that for
every finite sub-collection of ${\cal C}$, i.e., $(x_{1},t_{1})$,
$(x_{2},t_{2}),$..., $(x_{N},t_{N})\in{\cal C}$, there exists $y\in X$
and transition times $\tau_{1},...,\tau_{N-1}\in[0,\tau]$ such that
for $s_{0}=\tau_{0}=0$ and $s_{j}=\sum_{i=1}^{j}t_{i}+\sum_{i=1}^{j-1}\tau_{i}$,
we have 
\[
f_{s_{j-1}+\tau_{j-1}}(y)\in B_{t_{j}}(x_{j},\rho)
\]
for $j\in\{1,2,...,N\}$. If ${\cal C}$ has specification at all
scales, then we say ${\cal C}$ has \textit{specification}. We say
that the flow has specification if the entire orbit space ${\cal C}=X\times[0,\infty)$
has specification. 
\end{defn}

\begin{defn}
[Bowen property] We say $\vp:X\to\R$ a continuous potential has
the \textit{Bowen property} on ${\cal C}\subset X\times[0,\infty)$
if there are $\ep,K>0$ such that for all $(x,t)\in{\cal C}$, we
have 
\[
\sup_{y\in B_{t}(x,\ep)}\left|\Phi(x,t)-\Phi(y,t)\right|\leq K
\]
where $\Phi(x,t)=\int_{0}^{t}\vp(f_{\tau}x)d\tau$.
\end{defn}

\begin{defn}
[Decomposition of orbit segments] A decomposition of $X\times[0,\infty)$
consists of three collections ${\cal P}$, ${\cal G}$, ${\cal S}\subset X\times[0,\infty)$
such that:\begin{enumerate}[font=\normalfont] 

\item There exist $p,g,s:X\times[0,\infty)\to\R$ such that for each
$(x,t)\in X\times[0,\infty)$, we have $t=p(x,t)+g(x,t)+s(x,t)$, 

\item $(x,p(x,t))\in{\cal P}$, $(f_{p(x,t)}x,g(x,t))\in{\cal G}$,
and $(f_{p(x,t)+g(x,t)}x,s(x,t))\in{\cal S}$.

\end{enumerate}
\end{defn}

In Section \ref{sec:Decompositions-for-finite}, we will give the
precise construction of a decomposition $({\cal P},{\cal G},{\cal S})$
and prove that such decomposition has required properties in subsequent
sections. Due to some technical reasons (see \cite{Climenhaga:2016ut}),
we need to work on collections slightly bigger than ${\cal P}$ and
${\cal S}$, namely, 
\begin{align*}
[{\cal P}] & :=\{(x,n)\in X\times\N:\ (f_{-s}x,n+s+t)\in{\cal P}\ {\rm for}\ {\rm some}\ s,t\in[0,1]\},
\end{align*}
and similarly for $[{\cal S}].$

The following three terms are the remaining pieces needed for stating
the Climenhaga-Thompson criteria.

\begin{defn}
For $x\in X$, $\ep>0$ and $\vp:X\to\R$ a potential \begin{enumerate}[font=\normalfont] 

\item The \textit{bi-infinite Bowen ball} $\G_{\ep}(x)$ is defined
as 
\[
\G_{\ep}(x):=\{y\in X:\ d(f_{t}x,f_{t}y)\leq\ep\ {\rm for}\ {\rm all}\ t\in\R\}.
\]

\item The set of \textit{non-expansive points at scale} $\ep$ is
defined as
\[
{\rm NE(\ep}):=\{x\in X:\ \G_{\ep}\nsubseteq f_{[-s,s]}(x)\ {\rm for\ }{\rm any}\ s>0\}
\]
where $f_{[a,b]}(x)=\{f_{t}x:\ t\in[a,b]\}.$

\item \textit{The pressure of obstructions to expansivity} for $\vp$
is defined as 
\[
P_{{\rm exp}}^{\perp}(\vp):=\lim_{\ep\to0}P_{\exp}^{\perp}(\vp,\ep)
\]
 where 
\[
P_{\exp}^{\perp}(\vp,\ep):=\sup\{h_{\mu}(f_{1})+\int\vp d\mu:\ \mu\in{\cal M}^{e}({\cal F})\ {\rm and\ }\mu({\rm NE}(\ep))=1\}
\]
 and ${\cal M}^{e}({\cal F})$ is the set of ${\cal F}-$invariant
ergodic probability measures on $X$. 

\end{enumerate}
\end{defn}

\begin{rem}
For uniform hyperbolic systems, ${\rm NE}(\ep)=\emptyset$ for $\ep$
sufficiently small; thus $P_{\exp}^{\perp}(\vp)=-\infty$. In other
words, the condition $P_{{\rm exp}}^{\perp}(\vp)<P(\vp)$ always holds
in Bowen's work \cite{Bowen:1974df}. 
\end{rem}

Finally, the following theorem is the Climenhaga-Thompson criteria
for the uniqueness of equilibrium states. We will use this theorem
to prove Theorem A in Section \ref{sec:Pressure-gap-and_Thm_A}.

\begin{thm}
\cite[Theorem A]{Climenhaga:2016ut}\label{thm:C-T_criteria} Let
$(X,{\cal F})$ be a flow on a compact metric space, and $\vphi:X\to\R$
be a continuous potential. Suppose that $P_{{\rm exp}}^{\perp}(\vphi)<P(\vphi)$
and $X\times[0,\infty)$ admits a decomposition $({\cal P},{\cal G},{\cal S})$
with the following properties: \begin{itemize}[font=\normalfont]

\item[(I)] ${\cal G}$ has specification;

\item[(II)] $\vphi$ has Bowen property on ${\cal G}$;

\item[(III)] $P([{\cal P}]\cup[{\cal S}],\vphi)<P(\vphi)$.

\end{itemize}

Then $(X,{\cal F},\vphi)$ has a unique equilibrium state $\mu_{\vphi}$.
\end{thm}

\begin{rem}
\cite[Proposition 4.19]{Climenhaga:2016ut} points out that the unique
equilibrium state $\mu_{\vp}$ derived from the above theorem is ergodic. 
\end{rem}

\subsection{Gurevich pressure }

In this subsection, we introduce another well-studied notion of pressure,
the Gurevich pressure, that is, the growth rate of weighted periodic
orbits. In the uniformly hyperbolic setting, the Gurevich pressure
is equal to the topological pressure. However, it is not the case
for non-uniformly hyperbolic systems (see \cite{Gelfert:2014hn} for
more details). To make above discussion more precise, we shall define
the following relevant terms.

In what follows, let $M$ be a Riemannian manifold, ${\cal F}=(f_{t})_{t\in\R}$
be the geodesic flow on $T^{1}M$, and $\vp:T^{1}M\to\R$ be a continuous
potential. We denote the set of closed regular geodesics with length
in the interval $(a,b]$ by ${\rm Per}_{R}(a,b]$. For $\g\in{\rm Per}_{R}(a,b]$,
we define 
\[
\Phi(\g):=\int_{\g}\vp=\int_{0}^{|\g|}\vp(f_{t}v)dt
\]
 where $v\in T^{1}M$ is tangent to $\g$ and $|\g|$ is the length
of $\g$. Given $t,\Delta>0$, we define 
\[
\Lambda_{\Reg,\Delta}^{*}(\vp,t):=\sum_{\g\in{\rm Per}_{R}(t-\D,t]}e^{\Phi(\g)}.
\]

\begin{defn}
[Gurevich pressure] Given $\Delta>0$,

\begin{enumerate}[font=\normalfont]

\item The \textit{upper regular} \textit{Gurevich pressure} $\overline{P}_{\Reg,\D}^{*}$
of $\vp$ is defined as 

\[
\overline{P}_{\Reg,\D}^{*}(\vp):=\limsup_{t\to\infty}\frac{1}{t}\log\Lambda_{\Reg,\Delta}^{*}(\vp,t).
\]

\item The\textit{ lower} \textit{regular Gurevich pressure} $\underline{P}_{\Reg,\D}^{*}$
of $\vp$ is defined as 
\[
\underline{P}_{\Reg,\D}^{*}(\vp):=\liminf_{t\to\infty}\frac{1}{t}\log\Lambda_{\Reg,\Delta}^{*}(\vp,t).
\]
When $\overline{P}_{{\rm Reg},\D}^{*}(\vp)=\underline{P}_{\Reg,\D}^{*}(\vp)$,
we call this value the \textit{regular Gurevich pressure }and denote
it by $P_{\Reg,\D}^{*}(\vp)$.

\end{enumerate}
\end{defn}

\begin{rem}
Our \textit{upper} regular Gurevich pressure $\overline{P}_{\Reg,\D}^{*}$
is the regular Gurevich pressure $P_{{\rm Gur},{\cal R}}$ used in
\cite{Gelfert:2014hn}. Indeed, using the same argument as in \cite{Gelfert:2014hn},
one can show that $\overline{P}_{\Reg,\D}^{*}$ is independent of
$\D>0$. However, to derive the equidistribution property, we need
to take the lower regular Gurevich pressure into account (see Proposition
\ref{prop:pressure_equal_equidistribution}).
\end{rem}

\begin{defn}
For a potential $\vp:T^{1}M\to\R$, we say $\mu$ is a \textit{weak{*}
limit of $\vp-$weighted regular periodic orbits}, if there exists
$\D>0$ such that 
\[
\mu=\lim_{t\to\infty}\frac{\sum_{\g\in{\rm Per}_{R}(t-\D,t]}e^{\Phi(\g)}\delta_{\g}}{\Lambda_{\Reg,\Delta}^{*}(\vp,t)}
\]
 where $\delta_{\g}$ is the normalized Lebesgue measure along a periodic
orbit $\g$.
\end{defn}

In his proof of the variational principle in \cite[Theorem 9.10]{Walters:2000vc},
Walters pointed out a way to construct equilibrium states through
periodic orbits.

\begin{prop}
\cite[Theorem 9.10]{Walters:2000vc} \label{prop:walters_eq_states}Given
$\D>0$, suppose there exists $\{t_{k}\}$ such that 
\[
\lim_{k\to\infty}\frac{1}{t_{k}}\Lambda_{\Reg,\Delta}^{*}(\vp,t_{k})=P(\vp)
\]
 and 
\[
\lim_{k\to\infty}\frac{\sum_{\g\in{\rm Per}_{R}(t_{k}-\D,t_{k}]}e^{\Phi(\g)}\delta_{\g}}{\Lambda_{\Reg,\Delta}^{*}(\vp,t_{k})}=\mu,
\]
 then $\mu$ is an equilibrium state.
\end{prop}

From the above observation we have:

\begin{prop}
\label{prop:pressure_equal_equidistribution}Given $\D>0$, suppose
$P_{\Reg,\D}^{*}(\vp)=P(\vp)$ and $\vp$ has a unique equilibrium
$\mu_{\vp}$, then $\mu_{\vp}$ is a weak{*} limit of $\vp-$weighted
regular closed geodesics.
\end{prop}

\section{Preliminaries of surfaces without focal points\label{sec:Preliminaries-of-geometry}}

\subsection{Geometry of Riemannian manifolds without focal points}

In this section, we recall relevant earlier results of manifolds without
focal points. These results can be found in \cite{Eberlein:1973hu,Pesin:1977vi,Eschenburg:1977kn,Burns:1983dw}.

Throughout this section $M$ denotes a closed $C^{\infty}$ Riemannian
manifold. The \textit{geodesics flow} ${\cal F}=(f_{t})_{t\in\R}$
over the unit tangent bundle $T^{1}M$ is the flow given by $f_{t}(v)=\dot{\g}_{v}(t)$
where $\gamma_{v}$ is the (unit speed) geodesic determined by the
initial vector $v\in T^{1}M$. Recall that for any Riemannian manifold,
we can naturally equip its tangent bundle with the \textit{Sasaki
metric. }In what follows, without stating specifically, the norm $||\cdot||$
always refers to the Sasaki metric. 

A \textit{Jacobi field} $J(t)$ along a geodesic $\g$ is a vector
field along $\g$ which satisfying the Jacobi equation: 
\[
J''(t)+K(\g(t))\dot{\g}(t)=0
\]
 where $K$ is the Gaussian curvature and $'$ denotes the covariant
derivative along $\g$. A Jacobi field is \textit{orthogonal} if both
$J$ and $J'$ are orthogonal to $\dot{\g}$ at some $t_{0}\in\R$
(and hence for all $t\in\R)$.

\begin{defn}
[No focal points] A Riemannian manifold $M$ has \textit{no focal
points} if for any initial vanishing Jacobi field $J(t)$, its length
$\left\Vert J(t)\right\Vert $ is strictly increasing. We say $M$
has\textit{ no conjugate points} if any non-zero Jacobi field has
at most one zero.
\end{defn}

\begin{rem}
There are other equivalent definitions for manifolds without focal
points, and many of their geometric features are introduced in \cite{doCarmo:2013tg}.
The following results are classical and relevant in our setting: \begin{enumerate}[font=\normalfont] 

\item Nonpositively curved $\subsetneq$ no focal points$\subsetneq$
no conjugate points.

\item One can find examples from each category above from \cite{Gulliver:1975gx},
as well as \cite{Gerber:2003go}, for examples in the above assertion.

\end{enumerate}
\end{rem}

It is a classical result that one can identify the tangent space of
$T^{1}M$ with the space of orthogonal Jacobi fields ${\cal J}$.
Moreover, one can use this relation to define three ${\cal F}$ invariant
bundles $E^{u},E^{c},$ and $E^{s}$ on $TT^{1}M$. To be more precise,
let us denote ${\cal J}(\g)$ the space of space of orthogonal Jacobi
fields along a geodesic $\g$. For any $v\in T^{1}M$, the identification
between $T_{v}T^{1}M$ and ${\cal J}(\g_{v})$ is given by 
\[
T_{v}T^{1}M\ni\xi\mapsto J_{\xi}\in{\cal J}(\g_{v})
\]
 where $J_{\xi}$ is the Jacobi field determined by the initial data
$\xi$. Moreover, we have 
\begin{equation}
\left\Vert df_{t}(\xi)\right\Vert ^{2}=\left\Vert J_{\xi}(t)\right\Vert ^{2}+\left\Vert J'_{\xi}(t)\right\Vert ^{2}.\label{eq:metric_jacobi_unit_tangent_bundle}
\end{equation}

We define ${\cal J}^{s}(\g)$ to be the space of \textit{stable (orthogonal)
Jacobi fields} as 
\[
{\cal J}^{s}(\g)=\{J(t)\in{\cal J}(v):\ \left\Vert J(t)\right\Vert \ {\rm is}\ {\rm bounded\ for}\ t\geq0\},
\]
and ${\cal J}^{u}(\g)$ to be the space of \textit{unstable (orthogonal)
Jacobi fields} as 
\[
{\cal J}^{s}(\g)=\{J(t)\in{\cal J}(v):\ \left\Vert J(t)\right\Vert \ {\rm is}\ {\rm bounded\ for}\ t\leq0\}.
\]
Using these two linear spaces of ${\cal J}(\g)$ and the identification,
we can define two subbundles $E^{s}(v)$ and $E^{u}(v)$ of $T_{v}T^{1}M$
as the following:

\begin{align*}
E^{s}(v):= & \{\xi\in T_{v}T^{1}M:\ J_{\xi}\in{\cal J}^{s}(v)\},\\
E^{u}(v):= & \{\xi\in T_{v}T^{1}M:\ J_{\xi}\in{\cal J}^{u}(v)\}.
\end{align*}

Last, we define $E^{c}(v)$ given by the flow direction.

\begin{defn}
[Rank]The \textit{rank} of a vector $v\in T^{1}M$ is the dimension
of the space of parallel Jacobi fields. We call $M$ is a \textit{rank
1} manifold if it has at least one rank 1 vector. 
\end{defn}

\begin{defn}
[Singular and Regular set]The \textit{singular set} $\Sing\subset T^{1}M$
is the set of vectors with rank greater than or equal to 2. The \textit{regular
set} $\Reg$ is the complement of $\Sing.$
\end{defn}

The following proposition summarizes known facts about manifolds with
no focal points.

\begin{prop}
\label{prop:no_focal_points} Let $M$ be a closed Riemannian manifold
without focal points. Then we have \begin{enumerate}[font=\normalfont] 

\item \cite{Hurley:1986km} The geodesic flow ${\cal F}$ is topologically
transitive.

\item\cite[Proposition 4.7, 6.2]{Pesin:1977vi} $\dim E^{u}(v)=\dim E^{s}(v)=n-1$,
and $\dim E^{c}(v)=1$ where $\dim M=n$. 

\item\cite[Theorem 4.11, 6.4]{Pesin:1977vi} The subbundles $E^{u}(v)$,
$E^{s}(v)$, $E^{cu}(v)$ and $E^{cs}(v)$ are ${\cal F}$--invariant
where $E^{cs}(v)=E^{c}(v)\oplus E^{s}(v)$ and $E^{cu}(v)=E^{c}(v)\oplus E^{u}(v)$.

\item\cite[Theorem 6.1, 6.4]{Pesin:1977vi} The subbundles $E^{u}(v)$,
$E^{s}(v)$, $E^{cu}(v)$ and $E^{cs}(v)$ are integrable to ${\cal F}$--invariant
foliations $W^{u}(v)$, $W{}^{s}(v)$, $W^{cu}(v)$ and $W^{cs}(v)$,
respectively. Moreover, $W^{u}(v)$ (resp. $W^{s}(v)$) consists of
vectors perpendicular to $H^{u}(v)$ (resp. $H^{s}(v)$) and toward
to the same side as $v$ (see below for the definition of the horospheres
$H^{s/u}(v)$). 

\item\cite[Lemma, p. 246]{Eschenburg:1977kn} $E^{u}(v)\cap E^{s}(v)\neq\emptyset$
if and only if $v\in{\rm Sing}$.

\item \cite[Theorem 1]{OSullivan:1976cf},\cite[Theorem 2]{Eschenburg:1977kn}
\label{prop:the flat strip theorem}The Flat Strip Theorem: suppose
$M$ is simply connected and geodesics $\g_{1},\g_{2}$ are bi-asymptotic
in the sense that $d(\g_{1}(t),\g_{2}(t))$ is uniformly bounded for
all $t\in\R$. Then $\g_{1}$ and $\g_{2}$ bound a strip of flat
totally geodesically immersed surface.

\item\cite[Corollary 3.3, 3.6]{Eberlein:1973hu} Suppose $\dim M=2$,
then $v\in\Sing$ if and only if $K(\pi f_{t}v)=0$ for all $t\in\R$
where $\pi:T^{1}M\to M$ is the canonical projection. 

\item \cite{Hopf:1948tn} Suppose $\dim M=2$, then $M$ is rank
1 if and only if its genus is at least 2.

\end{enumerate}
\end{prop}

We shall introduce more metrics on $T^{1}M$ and the flow invariant
foliations induced in Proposition \ref{prop:no_focal_points} so that
we can perform finer analysis. We write $d_{{\rm S}}$ for the distance
function on $T^{1}M$ induced by the Sasaki metric on $TT^{1}M$.
Another handy distance function $d_{K}$ on $T^{1}M$, the \textit{Knieper
metric}, was introduced by Knieper in \cite{Knieper:1998ht}: 
\[
d_{K}(v,w):=\max\{d(\g_{v}(t),\g_{w}(t)):\ t\in[0,1]\}.
\]

It is not hard to see, $d_{{\rm S}}$ and $d_{K}$ are uniformly equivalent.
Thus, we will primarily work with the Knieper metric $d_{K}$ throughout
the paper. In particular, any Bowen ball $B_{t}(v,\ep)$ appearing
from here onward is with respect to the Knieper metric $d_{K}$, i.e.,
\[
B_{T}(v,\ep):=\{w\in T^{1}M:\ d_{K}(f_{t}w,f_{t}v)<\ep\ {\rm for}\ \mbox{{\rm all}\ }0\leq t\leq T\}.
\]

Furthermore, an \textit{intrinsic metric} $d^{s}$ on $W^{s}(v)$
for all $v\in T^{1}M$ is given by 
\[
d^{s}(u,w):=\inf\{l(\pi\g):\ \g:[0,1]\to W^{s}(v),\ \g(0)=u,\ \g(1)=w\}
\]
 where $l$ is the length of the curve in $M$, and the infimum is
taken over all $C^{1}$ curves $\g$ connecting $u,w\in W^{s}(v)$.
Using $d^{s}$ we can define the \textit{local stable leaf} through
$v$ of size $\rho$ as: 
\[
W_{\rho}^{s}(v):=\{w\in W^{s}(v):\ d^{s}(v,w)\leq\rho\}.
\]
 Moreover, we can locally define a similar intrinsic metric $d^{cs}$
on $W^{cs}(v)$ as: 
\[
d^{cs}(u,w)=|t|+d^{s}(f_{t}u,w)
\]
where $t$ is the unique time such that $f_{t}u\in W^{s}(w)$. This
metric $d^{cs}$ extends to the whole central stable leaf $W^{cs}(v)$.
We also define $d^{u},W_{\rho}^{u}(v),d^{cu}(v)$ analogously. Notice
that when $\rho$ is small these intrinsic metrics are uniformly equivalent
to $d_{{\rm S}}$ and $d_{K}$. A handy feature of these metrics is
that for $v\in T^{1}M$, $\sigma\in\{s,cs\}$ and for any $u,w\in W^{\sigma}$
the $t\mapsto d^{\sigma}(f_{t}u,f_{t}w)$ is a non-increasing function.
Indeed, it follows from the definition of manifolds with no focal
points. Similarly, for $\sigma\in\{u,cu\}$, $t\mapsto d^{\sigma}(f_{t}u,f_{t}w)$
is non-decreasing.

Following Proposition \ref{prop:no_focal_points}, one can define
the \textit{stable horosphere} $H^{s}(v)\subset M$ and the \textit{unstable
horosphere} $H^{u}(v)\subset M$ as the projection of the respective
foliations to $M$:
\[
H^{s}(v)=\pi(W^{s}(v))\mbox{\ and\ }H^{u}(v)=\pi(W^{s}(v)).
\]

We now summarize some useful properties of them.

\begin{prop}
\cite[Theorem 1 (i) (ii)]{Eschenburg:1977kn}\label{prop:horoshpere}
Let $M$ be a Riemannian closed manifolds without focal points. Then
we have \begin{enumerate}[font=\normalfont] 

\item$H^{u}(v)$, $H^{s}(v)$ are $C^{2}$-embedded hypersurfaces
when lifted to the universal cover $\widetilde{M}$. 

\item For $\sigma\in\{s,u\}$, the symmetric linear operator of ${\cal U}^{\sigma}(v):T_{\pi v}H^{\sigma}(v)\to T_{\pi v}H^{\sigma}(v)$
given by $v\mapsto\nabla_{v}N$, i.e., the shape operator on $H^{\sigma}(v)$,
is well-defined, where $N$ is the unit normal vector field on $H^{\sigma}(v)$
toward the same side as $v$ . 

\item${\cal U}^{u}$ is positively semidefinite and and ${\cal U}^{s}$
is negatively semidefinite.

\end{enumerate}
\end{prop}

We are ready to rephrase above two propositions specific to the surface
setting. From now on, we denote $S$ a closed Riemannian surface of
genus greater than or equal to 2 and has no focal point. Then from
Proposition \ref{prop:no_focal_points} and \ref{prop:horoshpere}
we have:  

\begin{itemize}

\item $S$ is rank 1.

\item For $v\in T^{1}S$, $H^{u}(v)$ (resp., $H^{s}(v)$) is one
dimensional and called the \textit{unstable} (resp., \textit{stable})
\textit{horocycle}. 

\item The (one dimensional) linear operator ${\cal U}^{u}(v):T_{\pi v}H^{u}(v)\to T_{\pi v}H^{u}(v)$
is the given by the \textit{geodesic curvature} $k^{u}(v)$ of the
horocycle of $H^{u}(v)$ at $\pi v$. More precisely, for all $w\in T_{\pi v}H^{u}(v)$
\[
{\cal U}^{u}(v)(w)=k^{u}(v)(w).
\]

\item Similarly, ${\cal U}^{s}(v)$ is given by $k^{s}(v)$ the \textit{geodesic
curvature} $k^{s}(v)$ of the horocycle of $H^{s}(v)$ at $\pi v$,
i.e., ${\cal U}^{s}(v)(w)=-k^{s}(v)(w)$ for all $w\in T_{\pi v}H^{s}(v)$.
Moreover, we have $k^{s}(-v)=k^{u}(v)$. 

\end{itemize}

\subsection{Hyperbolicity indices $\lambda$ and $\lt$\label{sub:Hyperbolicity-indices-}}

In this subsection, using $k^{s}$ and $k^{u}$ we introduce several
useful functions to quantify the hyperbolicity for any $v\in T^{1}S$.
These hyperbolicity indices will be used in Section \ref{sec:Decompositions-for-finite}
to derive the decomposition for orbit segments.

\begin{defn}
For $v\in T^{1}S$ and for any $T>0$, we define:

\begin{enumerate}[font=\normalfont] 

\item $\lambda(v):=\min(k^{u}(v),k^{s}(v)).$

\item$\lambda_{T}(v):=\int_{-T}^{T}\lambda(f_{\tau}v)d\tau$.

\end{enumerate}
\end{defn}

\begin{rem}
$\ $\begin{enumerate}[font=\normalfont]

\item Since the horocycles are $C^{2}$ (by Proposition \ref{prop:horoshpere}),
we have $k^{s}$ and $k^{u}$ are continuous, and so are $\lambda$
and $\lt$. 

\item The $\lambda$ defined in this paper is exactly the same as
the $\lambda$ introduced in \cite{BCFT2017}.

\end{enumerate}
\end{rem}

\begin{rem}
$\negmedspace$The most significant difference between the ``nonpositively
curved'' setting in \cite{BCFT2017} and our ``no focal points''
setting is that the norm of the Jacobi fields are not strictly convex
in the no focal points setting. As one can observe in \cite{BCFT2017},
the convexity of Jacobi fields implies good estimates on $\lambda$
so that one can use $\lambda$ to characterize the singular set. However,
in our setting, $\lambda$ does not enjoy such properties. Therefore,
we need to accumulate more hyperbolicity through integrating $\lambda$
for a longer time $T$; indeed, this is our motivation for introducing
a new function $\lambda_{T}$.
\end{rem}

The following proposition and lemma establish relations between horocycles
and related Jacobi fields. The version we state below is from \cite{BCFT2017}.

\begin{prop}
\label{prop:k^u_and_Coddazi_eq}Let $\g_{v}(t)$ be a unit speed geodesic
such that $\dot{\g}_{v}(0)=v$, and $J^{u}$ be the $H^{u}(v)-$Jacobi
field along $\g_{v}$, that is, the Jacobi field derived by varying
through geodesics perpendicular to $H^{u}(v)$. Then $J^{u}\in{\cal J}^{u}$
and $(J^{u})'(t)=k^{u}(f_{t}v)J^{u}(t)$ for all $t\in\R$. Similarly,
for $J^{s}$ the $H^{s}(v)-$Jacobi field along $\g_{v}$, we have
$J^{s}\in{\cal J}^{s}$ and $(J^{s})'(t)=-k^{s}(f_{t}v)J^{s}(t)$
for all $t\in\R$.
\end{prop}

\begin{proof}
Let $\alpha(s,t)$ for $(s,t)\in(-\ep,\ep)\times\R$ be the variation
of geodesics along $H^{u}(v)$, i.e., $\alpha(0,t)=\g_{v}(t)$ and
$\alpha(s,0)\in H^{u}(v)$, such that $\left.\frac{\partial}{\partial s}\alpha(s,t)\right|_{s=0}=J^{u}(t)$.
Then when $t=0$ 
\begin{align*}
(J^{u})'(0)= & \left.\frac{\nabla}{\partial t}\frac{\partial}{\partial s}\alpha(s,t)\right|_{s=0,t=0}=\left.\frac{\nabla}{\partial s}\frac{\partial}{\partial t}\alpha(s,t)\right|_{s=0,t=0}\\
= & \nabla_{J^{u}(0)}N={\cal U}^{u}(v)(J(0))=k^{u}(v)J(0)
\end{align*}
where the second equality is by the symmetry of the Levi-Civita connection
and the last equality follows from Proposition \ref{prop:horoshpere}.

To see this is true for all $t$, we notice that the flow invariant
unstable manifold $W^{u}(v)$ consists of vectors which are perpendicular
to $H^{u}(v)$ and point toward to the same side as $v$ (cf. Proposition
\ref{prop:no_focal_points}). That is, when we vary geodesics perpendicularly
along $H^{u}(v),$ these geodesics vary perpendicularly along $H^{u}(f_{t}v)$
as well. Thus, $J^{u}(t)$ is the Jacobi field derived by varying
geodesics perpendicular to $H^{u}(f_{t}v)$, and we have $(J^{u})'(t)=k^{u}(f_{t}v)J^{u}(t)$
by repeating the computation above. For $J^{s}$, the same argument
applies.
\end{proof}

Let $\Lambda$ be the maximum eigenvalue of $k^{u}(v)$ over all $v\in T^{1}S$.
Then for $\sigma\in\{s,u\}$ we have $\left\Vert (J^{\sigma})'(t)\right\Vert \leq\Lambda\left\Vert J^{\sigma}(t)\right\Vert $
for all $t$. By equation (\ref{eq:metric_jacobi_unit_tangent_bundle})
and the above proposition, for $\xi\in E^{u}(v)$ or $E^{s}(v)$ we
have 
\[
||J_{\xi}(t)||^{2}\leq||df_{t}\xi||^{2}\leq(1+\Lambda^{2})||J_{\xi}(t)||^{2}.
\]

\begin{lem}
\cite[Lemma 2.11]{BCFT2017} \label{lem:2.9}Let $v\in T^{1}S$ and
$J^{u}$ (resp. $J^{s}$) be an unstable (resp. stable) Jacobi field
along $\gamma_{v}$. Then, for , 
\begin{equation}
||J^{u}(t)||\geq e^{\int_{0}^{t}k^{u}(f_{\tau}v)d\tau}||J^{u}(0)||{\rm \ }and\ ||J^{s}(t)||\leq e^{-\int_{0}^{t}k^{s}(f_{\tau}v)d\tau}||J^{s}(0)||.
\end{equation}
 
\end{lem}

A handy lemma for computation:

\begin{lem}
\label{alglem} Let $\psi:\mathbb{R}\rightarrow\mathbb{R}$ be a non-negative
integrable function and
\[
\mbox{\ensuremath{\psi}}_{T}(t):=\int_{-T}^{T}\psi(t+\tau)d\tau.
\]
 Then, for every $a\leq b$, 
\[
\int_{a}^{b}\psi_{T}(t)dt\leq2T\int_{a-T}^{b+T}\psi(t)dt.
\]
Moreover, we have 
\[
\frac{1}{2T}\int_{0}^{t}\lambda_{T}(f_{\tau}v)d\tau-2T\Lambda\leq\int_{0}^{t}\lambda(f_{\tau}v)d\tau
\]
where $\Lambda=\max_{v\in T^{1}S}\lambda(v).$\end{lem}
\begin{proof}
For $b-a\leq2T$, 
\begin{eqnarray*}
\int_{a}^{b}\psi_{T}(t)dt & = & \int_{a}^{b}\int_{-T}^{T}\psi(t+\tau)d\tau dt\\
 & = & \int_{a-T}^{b-T}(\tau+T-a)\psi(\tau)d\tau+\int_{b-T}^{a+T}(b-a)\psi(\tau)d\tau+\int_{a+T}^{b+T}(b+T-\tau)\psi(\tau)d\tau\\
 & \leq & (b-a)\int_{a-T}^{b-T}\psi(\tau)d\tau+(b-a)\int_{b-T}^{a+T}\psi(\tau)d\tau+(b-a)\int_{a+T}^{b+T}\psi(\tau)d\tau\\
 & = & (b-a)\int_{a-T}^{b+T}\psi(\tau)d\tau\leq2T\int_{a-T}^{b+T}\psi(\tau)d\tau.
\end{eqnarray*}
For $b-a\geq2T$, 
\begin{eqnarray*}
\int_{a}^{b}\psi_{T}(t)dt & = & \int_{a}^{b}\int_{-T}^{T}\psi(t+\tau)d\tau dt\\
 & = & \int_{a-T}^{a+T}(\tau+T-a)\psi(\tau)d\tau+\int_{a+T}^{b-T}2T\psi(\tau)d\tau+\int_{b-T}^{b+T}(s+T-\tau)\psi(\tau)d\tau\\
 & \leq & 2T\int_{a-T}^{a+T}\psi(\tau)d\tau+2T\int_{a+T}^{b-T}\psi(\tau)d\tau+2T\int_{b-T}^{b+T}\psi(\tau)d\tau=2T\int_{a-T}^{b+T}\psi(\tau)d\tau.
\end{eqnarray*}

The last assertion follows from 

\begin{eqnarray*}
\int_{0}^{t}\lambda(f_{\tau}v)d\tau & = & \int_{-T}^{T+t}\lambda(f_{\tau}v)d\tau-\int_{-T}^{0}\lambda(f_{\tau}v)d\tau-\int_{t}^{T+t}\lambda(f_{\tau}v)d\tau\\
 & \geq & \int_{-T}^{T+t}\lambda(f_{\tau}v)d\tau-2T\Lambda\\
 & \geq & \frac{1}{2T}\int_{0}^{t}\lambda_{T}(f_{\tau}v)d\tau-2T\Lambda.
\end{eqnarray*}

\end{proof}

\section{A decompositions of finite orbit segments\label{sec:Decompositions-for-finite}}

We retain the same notations as previous sections and denote a potential
by $\vp:T^{1}S\to\R$.

\subsection{$\Sing$, $\lambda,\lt$, and decompositions}

In this subsection, we discuss a decomposition given by $\lt$. This
decomposition will allow us to apply the Climenhaga-Thompson criteria
(i.e., Theorem \ref{thm:C-T_criteria}) to prove the uniqueness of
equilibrium states.

\begin{defn}
[Good orbits and Bad orbits]\label{def:decomposition}For any $T,\eta>0$,
we define the two collections of finite orbit segments ${\cal G}_{T}(\eta),{\cal B}_{T}(\eta)\subset T^{1}S\times[0,\infty)$
using $\lambda_{T}$: 
\begin{align*}
{\cal G}_{T}(\eta):= & \{(v,t)\colon\ \int_{0}^{\tau}\lt(f_{\theta}v)d\theta\geq\tau\eta{\rm \ {\rm and\ }}\int_{0}^{\tau}\lt(f_{-\theta}f_{t}v)d\theta\geq\tau\eta\ \forall\tau\in[0,t]\},\\
{\cal B}_{T}(\eta):= & \{(v,t)\subset T^{1}S\times[0,\infty)\colon\ \int_{0}^{t}\lt(f_{\theta}v)d\theta<t\eta\}.
\end{align*}

\end{defn}

Using ${\cal G}_{T}$ and ${\cal B}_{T}(\eta)$, one can define a
orbit decomposition $({\cal P},{\cal G},{\cal S})=({\cal B}_{T}(\eta),{\cal G}_{T}(\eta),{\cal B}_{T}(\eta))$.
More precisely, we define three maps $p,g,s:T^{1}S\times[0,\infty)\to\R$
as follows. For any $(v,t)$ is a finite orbit segment, $p=p(v,t)$
is the largest time such that $(v,p)\in{\cal B}_{T}(\eta)$, $s=s(v,t)\in[0,t-p]$
is the largest time such that $(f_{t-s}v,s)\in{\cal B}_{T}(\eta)$,
and $g=g(v,t)=t-s-p$ is the remaining time in the middle. It is not
hard to see $(v,p)\in{\cal B}_{T}(\eta),$ $(f_{p}v,g)\in{\cal G}_{T}(\eta)$,
and $(f_{p+s}v,s)\in{\cal B}_{T}(\eta).$

\begin{figure}[H]

\begin{center}\includegraphics[scale=0.7]{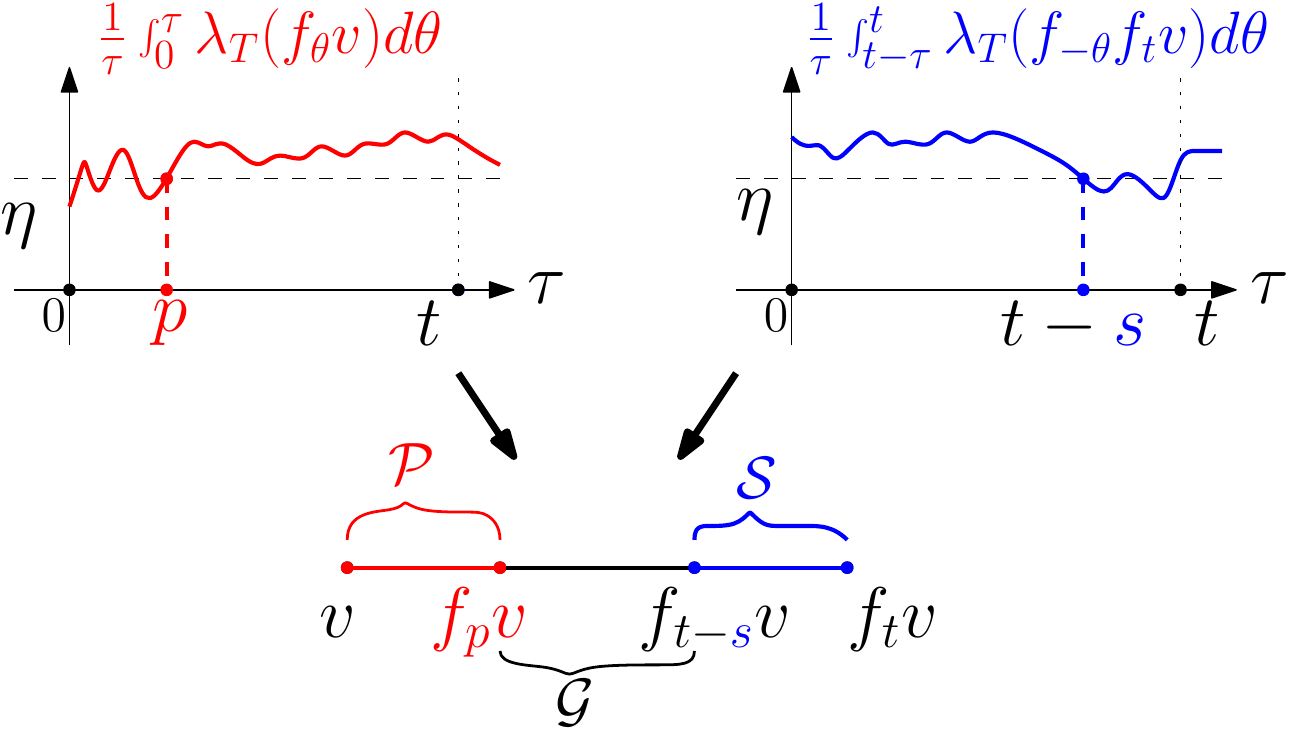}\end{center}

\caption{Orbit Decomposition}

\end{figure}

\begin{prop}
We have: \begin{enumerate}[font=\normalfont]

\item ${\rm Sing}$ is closed and flow invariant.

\item $\gt(\eta)\subset T^{1}S\times\mathbb{R}$ is closed.

\item ${\rm Reg}$ is dense in $T^{1}S$. 

\end{enumerate}
\end{prop}

\begin{proof}
These assertions are rather straightforward from their definitions
(notice that $\lt$ is continuous). Nevertheless, we elaborate a little
more on the last one since it is less obvious than others. Notice
that the geodesic flow is topologically transitive (see Proposition
\ref{prop:no_focal_points}), so there exists a dense orbit $\g\subset T^{1}S$.
Since $\Reg$ is an open set, there exists $t\in\R$ such that $\g(t)\in\Reg$,
and which implies that $\g\subset\Reg$ because $\Sing$ is flow invariant. 
\end{proof}

\subsection{Uniform estimates on ${\cal G}_{T}(\eta)$}

From the compactness of $T^{1}S$, the function $k^{s},k^{u},\lambda$
and $\lt$ are uniformly continuous, so given $\frac{\eta}{2}>0$,
there exists $\delta=\delta(\eta)$ such that when $d_{K}(v,w)<\delta$,
we have 
\[
|\Theta(v)-\Theta(w)|<\frac{\eta}{2}
\]
 where $\Theta:T^{1}S\to\R$ is one of $\lambda,$ $\lt$, $k^{u}$
and $k^{s}$. Also, define $\widetilde{\Theta}(v)=\max\{0,\Theta(v)-\frac{\eta}{2}\}$.
Then we have for $w\in B_{t}(v,\delta)$, then 

\begin{align}
\int_{0}^{t}\Theta(f_{\tau}w)d\tau\geq\int_{0}^{t}\widetilde{\Theta}(f_{\tau}v)d\tau & \geq\int_{0}^{t}\Theta(f_{\tau}v)d\tau-\frac{\eta}{2}t\label{eq:k_T_and_k_T_bar}
\end{align}
where, again, $\Theta\in\{\lambda,\lt,k^{s},k^{u}\}$.

Repeat the computation in Lemma \ref{alglem}, we have 
\begin{equation}
\int_{0}^{t}\widetilde{\lambda}(f_{\tau}v)d\tau\geq\frac{1}{2T}\int_{0}^{t}\widetilde{\lt}(f_{\tau}v)d\tau-2T\cdot\max\{0,\Lambda-\frac{\eta}{2}\}.\label{eq:k_T_tilte_k_tilte}
\end{equation}

Lastly, using the above notation, we have the following control of
the expansion and contraction along stable and unstable leaves.

\begin{lem}
\cite[Lemma 3.10]{BCFT2017} \label{lem:3.8BCFT}For any $\eta>0$,
pick $\delta=\delta(\eta)$ as above, $v\in T^{1}S$, and $w,w'\in W_{\delta}^{s}(v)$,
we have the following for every $t\geq0$: 
\[
d^{s}(f_{t}w,f_{t}w')\leq d^{s}(w,w')e^{-\int_{0}^{t}\widetilde{\lambda}(f_{\tau}v)d\tau}.
\]
Similarly, if $w,w'\in W_{\delta}^{u}(v)$, then for any $t\geq0$,
\[
d^{u}(f_{-t}w,f_{-t}w')\leq d^{u}(w,w')e^{-\int_{0}^{t}\widetilde{\lambda}(f_{-\tau}v)d\tau}.
\]

\end{lem}

The following lemma refines Lemma \ref{lem:3.8BCFT}. In other words
it provides us a nice control on the expansion and contraction for
orbit segments in $\gt$.

\begin{lem}
\label{lem:39}\label{lem:312} For any $T,\eta>0$, pick $\delta=\delta(\eta)$
as in ${\rm Lemma}$ \ref{lem:3.8BCFT}, and suppose $(v,t)\in\gt(\eta)$.
Then every $v'\in B_{t}(v,\delta)$ satisfies $(v',t)\in\gt(\frac{\eta}{2})$.
Moreover, there exists $C=C(T,\eta)>0$ such that for any $(v,t)\in\gt(\eta)$,
any $w,w'\in W_{\delta}^{s}(v)$, and any $0\leq\tau\leq t$, 
\[
d^{s}(f_{\tau}w,f_{\tau}w')\leq Cd^{s}(w,w')e^{-\frac{\eta}{4T}\tau}.
\]
Similarly, for $w,w'\in f_{-t}W_{\delta}^{u}(f_{t}v)$ and $0\leq\tau\leq t$,
we have 
\[
d^{u}(f_{\tau}w,f_{\tau}w')\leq Cd^{u}(f_{t}w,f_{t}w')e^{-\frac{\eta}{4T}(t-\tau)}.
\]
\end{lem}
\begin{proof}
By Lemma \ref{alglem}, \ref{lem:3.8BCFT} and equation (\ref{eq:k_T_and_k_T_bar})
and (\ref{eq:k_T_tilte_k_tilte}), because $v\in\gt(\eta)$ we have
\begin{align*}
d^{s}(f_{\tau}w,f_{\tau}w') & \leq d^{s}(w,w')e^{-\int_{0}^{\tau}\widetilde{\lambda}(f_{x}v)dx}\\
 & \leq d^{s}(w,w')\exp\left(\left(\frac{-1}{2T}\int_{0}^{\tau}\widetilde{\lt}(f_{x}v)dx\right)+2T\max\{0,\Lambda-\frac{\eta}{2}\}\right)\\
 & \leq d^{s}(w,w')\exp\left(\frac{-1}{2T}\left(\underset{\geq\tau\eta}{\underbrace{\int_{0}^{\tau}\lt(f_{x}v)dx}}-\frac{\eta}{2}\tau\right)+2T\max\{0,\Lambda-\frac{\eta}{2}\}\right)\\
 & \leq d^{s}(w,w')\exp\left(\frac{-\eta\tau}{2T}+\frac{\eta\tau}{4T}+2T\max\{0,\Lambda-\frac{\eta}{2}\}\right)=C\cdot d^{s}(w,w')e^{\frac{-\eta\tau}{4T}}
\end{align*}
where $C=e^{2T\max\{0,\Lambda-\frac{\eta}{2}\}}$.

Similarly, we have the other inequality.
\end{proof}

\begin{defn}
We define the uniformly regular set as ${\rm Reg}_{T}(\eta):=\{v\in T^{1}S:\ \lt(v)\geq\eta\}.$
\end{defn}

\begin{lem}
\label{lem: uniform angle} Given $\eta,T>0$, there exists $\theta>0$
so that for any $v\in{\rm Reg}_{T}(\eta)$, we have for any $-T\leq t\leq T$
\[
\measuredangle(E^{u}(f_{t}v),E^{s}(f_{t}v))\geq\theta.
\]
 \end{lem}
\begin{proof}
Assume the contrary. Then there exists $\{(v_{i},t_{i})\}_{i\in\mathbb{N}}\subset{\rm Reg}_{T}(\eta)\times[-T,T]$
such that 
\[
\measuredangle(E^{s}(f_{t_{i}}v_{i}),E^{u}(f_{t_{i}}v_{i}))\to0.
\]
Since ${\rm Reg}_{T}(\eta)\times[-T,T]$ is compact, there exist subsequences
$t_{i_{j}}\to t_{0}$, and $v_{i_{j}}\to v_{0}$ such that $\measuredangle(E^{s}(f_{t_{0}}v_{0}),E^{u}(f_{t_{0}}v_{0}))=0$.
Then, $f_{t_{0}}v_{0}\in{\rm Sing}$. On the other hand, ${\rm Reg}_{T}(\eta)$
is closed so $v_{0}\in{\rm Reg}_{T}(\eta)$. However, this is a contradiction
because $\text{Sing}$ is flow invariant. 
\end{proof}

\subsection{Relations between $k^{s},k^{u},\lambda,\lt$, and ${\rm Sing}$}

The aim of this subsection is to show how one can use these hyperbolicity
indices $\lambda$ and $\lambda_{T}$ to characterize the singular
set $\Sing$.

\begin{lem}
\label{lem:3.3}The following are equivalent for $v\in T^{1}S$. 

\begin{enumerate}[font=\normalfont]

\item $v\in{\rm Sing}$.

\item $k^{u}(f_{t}v)=0$ for all $t\in\R$.

\item $k^{s}(f_{t}v)=0$ for all $t\in\R$

\end{enumerate}
\end{lem}

\begin{proof}
It is clear that $(1)\implies(2)$ and $(3)$. We will prove $(2)\implies(1)$
which then $(3)\implies(1)$ similarly follows.

To see $(2)\implies(1)$, it is enough to show that $J^{u}$ the unstable
Jacobi field along $\g_{v}$ is parallel. By Proposition \ref{prop:k^u_and_Coddazi_eq},
we have for all $t\in\R$ 
\[
(J^{u})'(t)=k^{u}(f_{t}v)J^{u}(t)=0.
\]
Thus $J^{u}$ is a parallel Jacobi field.
\end{proof}

\begin{lem}
\label{lem:lamb_T_and_sing} $\lt(v)=0$ for all $T$ if and only
if $v\in\mathrm{Sing}$. 
\end{lem}

\begin{proof}
The if direction is clear. In the following we prove the only if direction. 

First we notice that since $\lambda$ is nonnegative, continuous,
we have that $\lt(v)=0$ for all $T\in\R$ implies $\lambda(f_{t}v)=0$
for all $t\in\R$.

\textbf{Claim:} There are only three possible cases such that $\lambda(f_{t}v)=0$
for all $t\in\R$: \begin{itemize}[font=\normalfont]

\item[(i)] $k^{s}(f_{t}v)=0$ for all $t\in\R$.

\item[(ii)] $k^{u}(f_{t}v)=0$ for all $t\in\R$.

\item[(iii)] There exists $t_{0}\in\R$ such that $k^{s}(f_{t_{0}}v)=k^{u}(f_{t_{0}}v)=0$.

\end{itemize}

It is clear from Lemma \ref{lem:3.3} that both $({\rm i})$ and $({\rm ii})$
give $v\in{\rm Sing}$. To see $({\rm iii})$ also implies $v\in{\rm Sing}$,
we recall that, for $\sigma\in\{s,u\}$, $k^{\sigma}(f_{t_{0}}v)=0$
implies that there exists $0\neq w^{\sigma}\in T_{\pi(f_{t_{0}}v)}H^{\sigma}(f_{t_{0}}v)$
such that $k^{\sigma}(w^{\sigma})=0$ . Since both $w^{u},w^{s}$
are tangent to $f_{t_{0}}v$ and $S$ is a surface, we know $w^{u}=w^{s}$
(by taking the same length). It is not hard to see that the $H^{u}(f_{t_{0}}v)-$
Jacobi field $J^{u}$ matches the $H^{s}(f_{t_{0}}v)-$ Jacobi field
$J^{s}$, that implies, $E^{u}(f_{t_{0}}v)\cap E^{s}(f_{t_{0}}v)\neq0$.
Thus we have $f_{t_{0}}v\in\Sing$, and because $\Sing$ is flow invariant
we have $v\in\Sing.$

To see the claim, let $U:=\{t\in\R:\ \lambda^{u}(f_{t}v)=0\}$ and
$W:=\{t\in\R:\ \lambda^{s}(f_{t}v)=0\}$. Since both $\lambda^{u},\lambda^{s}$
are continuous, $U$ and $W$ are closed sets in $\R$. Notice that
if $U\cap W=\emptyset$ then $U=\R\backslash W$; thus $U$, $W$
are clopen sets. Since $\R$ is connected, if $U\cap W=\emptyset$,
then $U=\R$ or $W=\R$.
\end{proof}

\begin{lem}
\label{lem:supp_in_sing}Let $\mu$ be a ${\cal F}$-invariant probability
measure on $T^{1}S$. Suppose $\lambda(v)=0$ for $\mu-$a.e. $v\in T^{1}S$,
then ${\rm supp(\mu)\subset\Sing}.$
\end{lem}

\begin{proof}
Suppose ${\rm supp}(\mu)\nsubseteq\Sing$. Since $\mu$ is Borel,
there exists $v\in{\rm Reg}\cap{\rm supp}(\mu)$ such that for any
$r>0$ we have $\mu(B(v,r))>0$. We also notice that since $v\in{\rm Reg}$
there exists $t_{0}$ such that $\lambda(f_{t_{0}}v)>0$ (otherwise
$v\in\Sing$ by Lemma \ref{lem:lamb_T_and_sing}). By the continuity
of $\lambda$, there exists a neighborhood $B(f_{t_{0}}v,r_{0})$
of $f_{t_{0}}v$ such that $\left.\lambda\right|_{B(f_{t_{0}}v,r_{0})}>0$.
Then there exists $r>0$ such $B(v,r)\subset f_{-t_{0}}(B(f_{t_{0}}v,r_{0}))$
and we have

\[
\mu(B(f_{t_{0}}v,r_{0})=\mu(f_{-t_{0}}(B(f_{t_{0}}v,r_{0})))\geq\mu(B(v,r))>0.
\]
Hence, $\lambda$ cannot vanish $\mu-$almost everywhere.
\end{proof}

\section{The specification property\label{sec:The-specification-property}}

Let $X$ be a compact metric space with metric $d$ and ${\cal F}=(f_{t})_{t\in\R}$
be a flow on $X$. For any $t\in\R^{+}$, we set $d_{t}(v,w)=\sup\limits _{s\in[0,t]}d(f_{s}v,f_{s}w)$
for any $v,w\in X$.

In what follows, $X$ will be $T^{1}S$ and $d$ the Knieper metric
$d_{K}$. With respect to the intrinsic metric $d^{cs}$ and $d^{u}$
on $W^{cs}$ and $W^{u}$, these metrics relate to each other by (from
the fact that the stable manifold is non-increasing in forward time)
\[
d_{K}(v,w)\leq d^{cs}(v,w)~~\text{ and }~~d_{K}(v,w)\leq e^{\Lambda}d^{u}(v,w)
\]
where $\Lambda=\max_{v\in T^{1}S}k^{s}(v)=\max_{v\in T^{1}S}k^{u}(v)$,
and which then imply 
\begin{align}
\begin{split}d_{t}(v,w) & \leq d^{cs}(v,w),\\
d_{t}(v,w) & \leq d^{u}(f_{t+1}v,f_{t+1}w)\leq e^{\Lambda}d^{u}(f_{t}v,f_{t}w).
\end{split}
\label{eq: metric comparison}
\end{align}

\begin{defn}
The foliations $W^{cs}$ and $W^{u}$ have \textit{local product structure}
at scale $\d>0$ with constant $\kappa\geq1$ at $v$ if for any $w_{1},w_{2}\in B(v,\d)$,
the intersection $[w_{1},w_{2}]:=W_{\k\d}^{u}(w_{1})\cap W_{\k\d}^{cs}(w_{2})$
is a unique point and satisfies 
\begin{align*}
d^{u}(w_{1},[w_{1},w_{2}]) & \leq\k d_{K}(w_{1},w_{2}),\\
d^{cs}(w_{2},[w_{1},w_{2}]) & \leq\k d_{K}(w_{1},w_{2}).
\end{align*}

\end{defn}

For any $T,\eta>0$, we define ${\cal C}_{T}(\eta):=\{(v,t)\colon v,f_{t}v\in\text{Reg}_{T}(\eta)\}$.
The uniform lower bound of $\lt$ on the endpoints of the orbits in
${\cal C}_{T}(\eta)$ guarantees the uniform local product structure
on ${\cal C}_{T}(\eta)$:

\begin{lem}
For any $T,\eta>0$, there exist $\d>0$ and $\k\geq1$ such that
${\cal C}_{T}(\eta)$ has local product structure at scale $\d$ with
constant $\k$. 
\end{lem}

\begin{proof}
The lemma follows from the uniform angle gap from Lemma \ref{lem: uniform angle}
together with the continuity of the distribution $E^{s}$ and $E^{u}$. 
\end{proof}
The following corollary is due to the transitivity of the geodesic
flow.

\begin{prop}
\label{cor: connecting orbits} Let $T,\eta>0$ be given. Then there
exists $\d>0$ such that for any $\rho\in(0,\d]$, there exists $a=a(\rho)$
such that the following holds: for any $v,w\in T^{1}M$ with $d_{K}(v,{\rm Reg}_{T}(\eta))<\delta$
and $d_{K}(w,{\rm Reg}_{T}(\eta))<\d$, there exists $\tau\in[0,a]$
and $[v,w]_{\tau}\in T^{1}S$ such that 
\begin{equation}
[v,w]_{\tau}\in W_{\rho}^{u}(v)~~\text{ and }~~f_{\tau}[v,w]_{\tau}\in W_{\rho}^{cs}(w).\label{eq: =00005B,=00005D_tau}
\end{equation}

\end{prop}

\begin{proof}
Let $\ep$ and $\kappa$ be the constants from the local product structure
on ${\rm Reg}_{T}(\eta)$. By taking $\d\in(0,\ep/2)$ sufficiently
small, we can ensure that the $\d$-neighborhood of ${\rm Reg}_{T}(\eta)$
has local product structure at scale $\ep/2$ with constant $2\kappa$.
Now using the transitivity of the flow $\F$, for any $\rho\in(0,\d)$,
we can find $a=a(\rho)$ such that the following holds: for any $v,w$,
there exists $x=x(v,w)\in B(v,\rho/4\kappa^{2})$ and $\tau\in(0,a)$
with $f^{\tau}x\in B(w,\rho/4\kappa^{2})$.

If $v,w$ happen to be $\d$-close to ${\rm Reg}_{T}(\eta)$, then
the uniform local product structure on $\d$-neighborhood of ${\rm Reg}_{T}(\eta)$
gives $[v,w]_{\tau}$ as follows: take $z=[v,x]$ and set $[v,w]_{\tau}:=f_{-\tau}[f_{\tau}z,w]$.
Then, $[v,w]_{\tau}:=f_{-\tau}[f_{\tau}z,w]$ satisfies \eqref{eq: =00005B,=00005D_tau}. 
\end{proof}

\begin{rem}
It is worth noting that the choices of $\tau$ and $[v,w]_{\tau}$
are not unique; we simply choose any one of $[v,w]_{\tau}$'s that
satisfy \eqref{eq: =00005B,=00005D_tau}. 
\end{rem}

\begin{prop}
\label{prop: specification} For any $\eta,T>0$, ${\cal C}_{T}(\eta)$
has specification. Hence, so does $\gt(\eta)$. 
\end{prop}

\begin{proof}
We begin by fixing any $(v_{0}',t_{0}')\in\gt(\eta)$ as our reference
orbit. To simplify the notation, we set $v_{0}:=f_{-T}v_{0}'$ and
$t_{0}:=2T+t_{0}'$. Then $(v_{0},t_{0})$ is just an extended orbit
segment obtained from of $(v_{0}',t_{0}')$. 

Using the uniform continuity of $\lambda$, we can choose $\d>0$
such that $|\lambda(v)-\lambda(w)|<\frac{\eta t_{0}'}{4Tt_{0}}$ whenever
$d_{K}(v,w)<\d$. For such choice of $\d$, for any $w\in B_{t_{0}}(v_{0},\d)$
we have 
\begin{align*}
2T\int_{0}^{t_{0}}\lambda(f_{s}w)ds & \geq2T\int_{0}^{t_{0}}\lambda(f_{s}v_{0})ds-(2T)t_{0}\cdot\frac{\eta t_{0}'}{4Tt_{0}},\\
 & \geq\eta t_{0}'-\frac{\eta t_{0}'}{2}=\frac{\eta t_{0}'}{2}.
\end{align*}
In particular, setting $\alpha:=\exp(\frac{\eta t_{0}}{4T})>1$, we
have that for any $w,w'\in f_{-t_{0}}W_{\d}^{u}(f_{t_{0}}v_{0})$
\begin{equation}
\alpha d^{u}(f_{t_{0}}w,f_{t_{0}}w')\leq d^{u}(w,w')
\end{equation}
 Given an arbitrary small scale $0<\rho<\min(\d,\ep)$, we will show
that we can set $\rho':=\rho/\big(6e^{\Lambda}\sum\limits _{i=1}^{\infty}\alpha^{-i}\big)$
so that ${\cal C}_{T}(\eta)$ has specification with scale $\rho$
with corresponding $\tau(\rho):=t_{0}+2a$ with $a:=a(\rho')$ from
Proposition \ref{cor: connecting orbits}.

Let $(v_{1},t_{1}),\ldots,(v_{n},t_{n})\in{\cal C}_{T}(\eta)$ be
given. We will inductively define orbit segments $(w_{j},s_{j})$
such that for each $1\leq j\leq n$, we have 
\begin{equation}
f_{s_{j}}w_{j}\in W_{\rho'}^{cs}(f_{t_{j}}v_{j}).\label{eq: inductive hyp}
\end{equation}

We begin by setting $(w_{1},s_{1}):=(v_{1},t_{1})$. Supposing that
$(w_{j},s_{j})$ satisfies \eqref{eq: inductive hyp}, we want to
define $(w_{j+1},s_{j+1})$ in a way that the orbit of $w_{j+1}$
closely shadows that of $w_{j}$ for time $s_{j}$, then jumps (via
Proposition \ref{cor: connecting orbits} with transition time $\leq a$)
to $v_{0}$ and shadows $v_{0}$ for time $t_{0}$, then jumps to
(again via Proposition \ref{cor: connecting orbits}) and shadows
$w_{j+1}$ for time $s_{j+1}$. 

Since Corollary only allows one jump at a time, we define an auxiliary
orbit segments 
\[
(u_{j},l_{j}):=\big(f_{-s_{j}}[f_{s_{j}}w_{j},v_{0}]_{\tau_{j}},s_{j}+\tau_{j}+t_{0}\big)
\]
by applying Proposition \ref{cor: connecting orbits} to $f_{s_{j}}w_{j}$
and $v_{0}$. Note that Corollary can be successfully applied because
$f_{s_{j}}w_{j}\in W_{\rho'}^{cs}(f_{t_{j}}v_{j})$ from \eqref{eq: inductive hyp}
and $f_{t_{j}}v_{j}\in{\rm Reg_{T}(\eta)}$ from $(v_{j},t_{j})\in{\cal C}_{T}(\eta)$.
Also, from its definition, the orbit segment $(u_{j},l_{j})$ satisfies
(1) and (2). Moreover, $f_{l_{j}}u_{j}\in W_{\rho'}^{cs}(v_{0})$
because $f_{s_{j}+\tau_{j}}u_{j}\in W_{\rho'}^{cs}(v_{0})$ and $d^{s}$
doesn't increase in forward time (quote definition of no focal points).

We then apply Proposition \ref{cor: connecting orbits} again to $f_{l_{j}}u_{j}$
and $v_{j+1}$ to obtain 
\[
(w_{j+1},s_{j+1}):=\big(f_{-l_{j}}[f_{l_{j}}u_{j},v_{j+1}]_{\tau_{j}'},l_{j}+\tau_{j}'+t_{j+1}\big).
\]
From the same reasoning as in the construction of $(u_{j},l_{j})$,
the new orbit segment $(w_{j+1},s_{j+1})$ is well-defined, $f_{s_{j+1}}w_{j+1}\in W_{\rho'}^{cs}(f_{t_{j+1}}v_{j+1})$,
and satisfies (1), (2), and (3). 

Now we show that $(w_{j},s_{j})$ constructed as above shadows each
$(v_{i},t_{i})$ up to $i=j$ with scale $\rho'$; that is, $d_{t_{i}}(f_{s_{i}-t_{i}}w_{j},v_{i})<\rho$.
From the construction, notice that for any $i\leq m\leq j$, we have
\[
d^{u}(f_{s_{i}}w_{m},f_{s_{i}}u_{m})\leq\rho'\alpha^{-(m-i)}.
\]
This is because $d^{u}(f_{s_{m}}w_{m},f_{s_{m}}u_{m})\leq\rho'$ from
the construction of $u_{m}$ and each time $f_{s_{m}}u_{m}$ and $f_{s_{m}}w_{m}$
pass through the reference orbit $(v_{0},t_{0})$ in backward time,
their $d^{u}$ distance decrease by a factor of at least $\alpha$.
Similarly, we have 
\[
d^{u}(f_{s_{i}}u_{m},f_{s_{i}}w_{m+1})\leq\rho'\alpha^{-(1+m-i)}.
\]
Hence, for any $i\leq j$, we can uniformly bound the $d^{u}$ distance
$d^{u}(f_{s_{i}}w_{j},f_{s_{i}}w_{i})$ by $\frac{\rho}{3e^{\Lambda}}$
: 
\begin{align*}
d^{u}(f_{s_{i}}w_{j},f_{s_{i}}w_{i}) & \leq\sum\limits _{m=i}^{j-1}d^{u}(f_{s_{i}}w_{m}.f_{s_{i}}w_{m+1}),\\
 & \leq\sum\limits _{m=i}^{j-1}d^{u}(f_{s_{i}}w_{m},f_{s_{i}}u_{m})+d^{u}(f_{s_{i}}u_{m},f_{s_{i}}u_{m},f_{s_{i}}w_{m+1}),\\
 & \leq\rho'\sum\limits _{m=i}^{j-1}\alpha^{-(m-i)}+\rho'\sum\limits _{m=i}^{j-1}\alpha^{-(1+m-i)},\\
 & \leq\frac{\rho}{3e^{\Lambda}},
\end{align*}
where the last inequality is due to the definition of $\rho'$. From
the relations among various metrics \eqref{eq: metric comparison},
we obtain that 
\begin{align*}
d_{t_{i}}(f_{s_{i}-t_{i}}w_{j},v_{i}) & \leq d_{t_{i}}(f_{s_{i}-t_{i}}w_{j},f_{s_{i}-t_{i}}w_{i})+d_{t_{i}}(f_{s_{i}-t_{i}}w_{i},v_{i}),\\
 & \leq\frac{\rho}{3e^{\Lambda}}\cdot e^{\Lambda}+\rho'\leq\rho,
\end{align*}
where we have used that $d^{s}(f_{s_{i}-t_{i}}w_{i},v_{i})\leq\rho'$
from the construction of $w_{i}$. Since $\rho$ was arbitrary, this
finishes the proof. 
\end{proof}

One useful corollary of the specification property is the closing
lemma which creates lots of periodic orbits, and later allows ${\cal C}_{T}(\eta)$
to be approximated by regular periodic orbits. The proof of the closing
lemma below follows the same idea as \cite[Lemma 4.7]{BCFT2017}.

\begin{lem}
[The closing lemma]\label{lem:closing_lem} For any given $T,\eta,\ep>0$,
there exists $s=s(\ep)>0$ such that for any $(v,t)\in{\cal C}_{T}(\eta)$
there exists $w\in B_{t}(v,\ep)$ and $\tau\in[0,s(\ep)]$ with $f_{t+\tau}w=w$. 
\end{lem}

\begin{proof}
The proof is based on Brouwer's fixed point theorem. We begin by fixing
$(v_{0}',t_{0}')\in\gt$ and set $(v_{0},t_{0}):=(f_{-T}v_{0}',2T+t_{0}')$
as in Proposition \ref{prop: specification}. Reasoning as in Proposition
\ref{prop: specification}, there exists $\d>0$ such that the distance
between any $w,w'\in W_{\d}^{s}(v_{0})$ contract (and likewise expand
for any $w,w'\in f_{-t_{0}}W_{\d}^{u}(f_{t_{0}}v_{0})$ under $f_{t_{0}}$)
by factor $\alpha:=\exp(\frac{\eta t_{0}}{4T})$.

Let $\ep=\ep_{0}/4$. We may suppose $\ep$ is small enough that ${\cal C}_{T}(\eta)$
has local product structure at scale $\ep$ and constant $\kappa$.
Let $n\in\N$ such that $\alpha^{n}>2\kappa$. Also, we may assume
$nt_{0}\geq1+\ep$ without loss of generality (otherwise, simply increase
$n$).

Now, for any $(v,t)\in{\cal C}_{T}(\eta)$, we use Proposition \ref{prop: specification}
to find $w_{0}\in B(v,\ep/4\kappa)$ whose orbits shadows $(v,t)$
once, then $(v_{0},t_{0})$ $n$-times, and then $(v,t)$ once again
at scale $\ep/4\kappa$ with each transition time bounded above by
$\widetilde{\tau}$. Since $w_{0}$ has to eventually shadow $(v,t)$
again, there exists $\tau_{0}\in[nt_{0},n(t_{0}+\widetilde{\tau})+\widetilde{\tau}]$
such that $f_{t+\tau_{0}}w_{0}\in B(v,\ep/4\kappa)$. From the triangle
inequality pivoted at $v$, we have $d_{K}(w_{0},f_{t+\tau_{0}}w_{0})<2\cdot\ep/4\kappa=\ep/2\kappa$.
Also, using the forward contraction of the stable manifold near the
reference orbit $(v_{0},t_{0})$, for any $u\in W_{\ep}^{s}(w_{0})$,
we have 
\begin{align*}
d_{K}(f_{t+\tau_{0}}u,w_{0}) & \leq d_{K}(f_{t+\tau_{0}}u,f_{t+\tau_{0}}w_{0})+d_{K}(f_{t+\tau_{0}}w_{0},w_{0}),\\
 & \leq\alpha^{-n}d_{K}(u,w_{0})+\ep/2\kappa\leq\ep/\kappa.
\end{align*}
Since $v$ has local product structure at scale $\ep$ with constant
$\kappa$ and $\w_{0}$ is $\ep/4\kappa$-close to $v$, the point
$W_{\ep}^{s}(w_{0})\cap W_{\ep}^{cu}(f_{t+\tau_{0}}u)$ is well-defined
and belongs to $W_{\ep}^{s}(w_{0})$. In particular, the continuous
map from $W_{\ep}^{s}(w_{0})$ to itself given by 
\[
u\mapsto W_{\ep}^{s}(w_{0})\cap W_{\ep}^{cu}(f_{t+\tau_{0}}u)
\]
is well-defined. Hence, by Brouwer fixed point theorem, we can find
a fixed point $w_{1}\in W_{\ep}^{s}(w_{0})$ under this map. Since
the map is not given by $f_{s}$ for some $s$, the fixed point $w_{1}$
is not quite $\F$ invariant yet. Instead, its characterizing property
is that $w_{1}\in W_{\ep}^{cu}(f_{t+\tau_{0}}w_{1})$.

By adjusting $\tau_{0}$ by a unique small constant less than $\ep$,
we have $w_{1}\in W_{\ep}^{u}(f_{t+\tau}w_{1})$ where $\tau$ is
adjusted constant from $\tau_{0}$. Since the unstable manifold shrinks
in backward time near $(v_{0},t_{0})$ by factor $\alpha$, this time
we obtain a continuous map defined by the flow $f_{-t-\tau}$: 
\[
f_{-t-\tau}\colon W_{2\ep}^{u}(f_{t+\tau}w_{1})\to W_{2\ep}^{u}(f_{t+\tau}w_{1}).
\]
Hence, the Brouwer fixed point theorem applies again and we obtain
$w\in W_{2\ep}^{u}(f_{t+\tau}w_{1})$ with $f_{t+\tau}w=w$. We are
left to show that $d_{t}(v,w)\leq\ep_{0}$. This follows because 
\begin{align*}
d_{t}(v,w) & \leq d_{t}(v,w_{0})+d_{t}(w_{0},w_{1})+d_{t}(w_{1},w),\\
 & \leq\ep/4\kappa+d^{s}(w_{0},w_{1})+d^{u}(w_{1},w),\\
 & \leq\ep/4\kappa+\ep+2\ep\leq\ep_{0}.
\end{align*}
Here, we have used \eqref{eq: metric comparison} and the fact that
$d_{t}(w_{1},w)\leq d^{u}(f_{t+1}w_{1},f_{t+1}w)\leq d^{u}(f_{t+\tau}w_{1},f_{t+\tau}w)$
because $\tau\geq nt_{0}-\ep\geq1$. Lastly, setting $s(\ep_{0}):=n(t_{0}+\widetilde{\tau})+\widetilde{\tau}+\ep$,
we are done. 
\end{proof}

Using the same argument as \cite[Corollary 4.8]{BCFT2017}, we have
the following corollary of the closing lemma.

\begin{cor}
\label{cor: closing_lemma_regular}For any given $T,\eta>0$, there
exist $\ep=\ep(T,\eta)>0$ such that for any $\ep_{0}<\ep$ there
exists $s=s(\ep_{0})>0$ satisfying the following: for any $(v,t)\in{\cal C}_{T}(\eta)$
there exists

\begin{enumerate}[font=\normalfont]

\item a regular vector $w$ with $w\in B_{t}(v,\ep_{0})$, and

\item $\tau\in[0,s]$ with $f_{t+\tau}w=w$.

\end{enumerate}\end{cor}
\begin{proof}
From the uniform continuity of $\lambda$, there exists $\ep=\ep(\eta)>0$
such that for all $w\in B(v,\ep)$, we have $\lambda(w)>0.$

Since $v\in{\cal C}_{T}(\eta)$, there exists $v'=f_{\sigma}v$ for
some $\sigma\in[-T,T${]} such that $\lambda(v')>\eta$. Also, we
must have $(v',t+\sigma)\in{\cal C}_{2T}(\eta)$ from the definition
of ${\cal C}_{T}(\eta)$. By Lemma \ref{lem:closing_lem}, for any
$2T,\eta,\ep_{0}>0$, there exists $s=s(\ep_{0})>0$ such that $w\in B_{t+\sigma}(v',\ep_{0})$
and $\tau\in[0,s(\ep_{0})]$ such that $f_{t+\sigma+\tau}(w)=w.$ 

Also, it follows that $w$ is a regular vector because $\lambda(w)>0$
from $d_{K}(v',w)<\ep_{0}<\ep.$
\end{proof}

\section{The Bowen property\label{sec:The-Bowen-property}}

In this section, we prove the Bowen property for H\"{o}lder potentials
and the geometric potential $\vp^{u}$. Lemmas in this section are
exactly the same as their corresponding lemmas in \cite{BCFT2017},
and proofs follow mutatis mutandis. However, for the completeness
we still give proofs there.

\subsection{The Bowen property for H\"{o}lder potentials}

\begin{defn}
A function $\vp:T^{1}S\to\R$ is called \textit{H\"{o}lder}\textit{
along stable leaves} if there exist $C,\theta,\delta>0$ such that
for $v\in T^{1}S$ and $w\in W_{\delta}^{s}(v)$, one has $|\vp(v)-\vp(w)|\leq Cd^{s}(v,w)^{\theta}$.
Similarly, $\vp$ is called \textit{H\"{o}lder}\textit{ along unstable
leaves }if there exist $C,\theta,\delta>0$ such that for $v\in T^{1}S$
and $w\in W_{\delta}^{u}(v)$, one has $|\vp(v)-\vp(w)|\leq Cd^{u}(v,w)^{\theta}$. 
\end{defn}

Since $d_{K}$ is equivalent to $d^{u}$ and $d^{s}$ along unstable
and stable leaves when $\delta$ is small, we know $\vp$ is\textit{
}H\"{o}lder if and only if $\vp$ is H\"{o}lder along stable and
unstable leaves.

\begin{defn}
A function $\vp$ is said to have the \textit{Bowen property along
stable leaves} with respect to ${\cal C}\subset T^{1}S\times[0,\infty)$
if there exist $\delta,K>0$ such that 
\[
\sup\left\{ |\Phi(v,t)-\Phi(w,t)|:\ (v,t)\in{\cal C},\ w\in W_{\delta}^{s}(v)\right\} \leq K.
\]
 Similarly, a function $\vp$ is said to have the \textit{Bowen property
along unstable leaves} with respect to ${\cal C}\subset T^{1}S\times[0,\infty)$
if there exist $\delta,K>0$ such that 
\[
\sup\left\{ |\Phi(v,t)-\Phi(w,t)|:\ (v,t)\in{\cal C},\ w\in f_{-t}W_{\delta}^{u}(f_{t}v)\right\} \leq K.
\]

\end{defn}

\begin{lem}
\label{lem:holder_bowen_unstable_stable}For any $T,\eta>0$, if $\varphi$
is H\"{o}lder along stable leaves (resp. unstable leaves), then $\varphi$
has the Bowen property along stable leaves (resp. unstable leaves)
with respect to $\gt(\eta)$. \end{lem}
\begin{proof}
It is a direct consequence of Lemma \ref{lem:312}. We prove the stable
leaves case, and for unstable leaves one uses the same argument.

Let $(v,t)\in\gt(\eta)$, $\delta_{1}>0$ be as in Lemma \ref{lem:312}
and $\delta_{2}>0$ be given by the H\"{o}lder continuity along stable
leaves. Then for $\delta=\min\{\delta_{1},\delta_{2}\}$ and $w\in W_{\delta}^{s}(v)$,
we have

\begin{align*}
\left|\Phi(v,t)-\Phi(w,t)\right| & \leq\int_{0}^{t}\left|\varphi(f_{\tau}v)-\varphi(f_{\tau}w)\right|{\rm d}\tau\leq\int_{0}^{t}C_{1}\cdot d^{s}(f_{\tau}v,f_{\tau}w)^{\theta}{\rm d}\tau\\
 & \leq\int_{0}^{t}C_{1}\cdot\left(Cd^{s}(v,w)\cdot e^{-\frac{\eta}{4T}\tau}\right)^{\theta}{\rm d}\tau\leq C_{1}\cdot C^{\theta}\cdot d^{s}(v,w)^{\theta}\int_{0}^{t}e^{\frac{-\eta\theta}{4T}\tau}{\rm d}\tau\\
 & \leq C_{1}\cdot C^{\theta}\delta^{\theta}\frac{4T}{\eta\theta}.
\end{align*}

\end{proof}

\begin{lem}
\label{lem:Bowen_stable_unstable_all}Given $T,\eta>0$, suppose $\varphi$
has the Bowen property along stable leaves and unstable leaves with
respect to $\gt(\frac{\eta}{2})$. Then $\varphi:T^{1}S\to\mathbb{R}$
has the Bowen property on $\gt(\eta)$. \end{lem}
\begin{proof}
We first notice that since the curvature of horocycles is uniformly
bounded, $d_{K}$ and $d^{u}$ are equivalent on $W_{\delta}^{u}$
when $\delta$ small enough. Hence, there exist $\delta_{0},C>0$
such that $d^{u}(u,v)\leq Cd_{K}(u,v)$ for $v\in T^{1}S$ and $u\in W_{\delta_{0}}^{u}(v).$
Let $\delta_{1}>0$ be the radius that guarantees for any $(v,t)\in{\cal G}_{T}(\eta)$
the foliations $W^{u}$ and $W^{cs}$ have local product structure
at scale $\delta_{1}$ with constant $\kappa$. Let $\delta_{2}>0$
be the radius given in Lemma \ref{lem:312} that if $(v,t)\in\gt(\eta)$,
then for $w\in W_{\delta_{2}}^{u}(v)$ or $w\in W_{\delta_{2}}^{s}$,
we have $(w,t)\in\gt(\frac{\eta}{2})$. Let $\delta_{3},K>0$ be the
constants from the Bowen property for $\varphi$ along stable and
unstable leaves with respect to $\gt(\frac{\eta}{2})$. Without loss
generality we may assume $\delta_{3}<\delta_{0}$.

Let $\delta=\min\{\delta_{0},\delta_{1},\delta_{2},\frac{\delta_{3}}{2\kappa C},\frac{\delta_{3}}{\kappa}\}$,
$(v,t)\in\gt(\eta)$, and $w\in B_{t}(v,\delta)$.

By the local product structure there exists unique $v'\in W_{\kappa\delta}^{u}\cap W_{\kappa\delta}^{cs}(v)$.
Suppose for now that $f_{t}(v')\in W_{\delta_{3}}^{u}(f_{t}w)$. Then
we find $\rho\in[-\kappa\delta,\kappa\delta]$ such that $f_{\rho}(v')\in W_{\kappa\delta}^{s}(v)\subset W_{\delta_{3}}^{s}(v)$.
Thus, from the the Bowen property along the stable and unstable leaves,
we have 
\[
\left|\Phi(v,t)-\Phi(f_{\rho}v',t)\right|\leq K{\rm \ and}\ \left|\Phi(v',t)-\Phi(w,t)\right|\leq K.
\]
Then,
\begin{align*}
\left|\Phi(v,t)-\Phi(w,t)\right| & \leq\left|\Phi(v,t)-\Phi(f_{\rho}v',t)\right|+\left|\Phi(f_{\rho}v',t)-\Phi(v',t)\right|+\left|\Phi(v',t)-\Phi(w,t)\right|\\
 & \leq2K+2||\varphi||\cdot|\rho|.
\end{align*}

Now, let us establish $f_{t}(v')\in W_{\delta_{3}}^{u}(f_{t}w)$.
Supposing that it does not hold, then there exists $\sigma\in[0,t]$
such that 
\begin{equation}
\delta_{3}<d^{u}(f_{\sigma}v',f_{\sigma}w)\leq\delta_{0}.\label{eq:7.1}
\end{equation}

Notice that $v'\in W_{\kappa\delta}^{cs}(v)\subset B_{t}(v,\kappa\delta)$,
so 
\[
d_{K}(f_{\sigma}v',f_{\sigma}w)\leq d_{K}(f_{\sigma}v',f_{\sigma}v)+d_{K}(f_{\sigma}v,f_{\sigma}w)\leq2\kappa\delta.
\]
Thus, $d^{u}(f_{\sigma}v',f_{\sigma}w)\leq2C\kappa\delta<\delta_{3}$
and we have derived a contradiction to (\ref{eq:7.1}). 
\end{proof}

Summing up two lemmas above, we have the desired result for H\"{o}lder
potentials:

\begin{thm}
\label{thm:Holder_Bowen}If $\varphi$ is H\"{o}lder continuous,
then it has the Bowen property with respect to $\gt(\eta)$ for any
$T,\eta>0$.
\end{thm}

\subsection{The Bowen property for the geometric potential }

\begin{defn}
The \textit{geometric potential} $\vp^{u}:T^{1}S\to\R$ is defined
as: for $v\in T^{1}S$ 
\[
\varphi^{u}(v):=-\lim_{t\to0}\frac{1}{t}{\displaystyle \log\det(}\left.df_{t}\right|_{E^{u}(v)})=-\left.\frac{d}{dt}\right|_{t=0}{\displaystyle \log\det(}\left.df_{t}\right|_{E^{u}(v)}).
\]

\end{defn}

In general, we do not know if $\vp^{u}$ is H\"{o}lder continuous.
There are some partial results under the nonpositively curved assumption;
however, not much is known in the no focal points setting. Nevertheless,
in this subsection we prove $\vp^{u}$ has the Bowen property on $\gt(\eta)$.

We denote by $J_{v}^{u}$ the unstable Jacobi field along $\gamma_{v}$
with $J_{v}^{u}(0)=1$. Let $U_{v}^{u}:=(J_{v}^{u})'/J_{v}^{u}$,
then $U_{v}^{u}$ is a solution to the Riccati equation 
\[
U'+U^{2}+K(f_{t}v)=0.
\]
Notice that we also have $U_{v}^{u}(t)=k^{u}(f_{t}v)$. Notice the
following lemma relates $\vp^{u}(t)$ and $-U_{v}^{u}(t)$.

\begin{lem}
\cite[Lemma 7.6]{BCFT2017} There exists a constant $K$ such that
for all $v\in T^{1}S$ and $t>0$ we have 
\[
\left|\int_{0}^{t}\vp^{u}(f_{\tau}v)d\tau-\int_{0}^{t}-U_{v}^{u}(\tau)d\tau\right|\leq K.
\]

\end{lem}

Hence, in order to prove the Bowen property of $\varphi^{u}$ on $\mathcal{G}_{T}(\eta)$,
we only have to prove Lemma \ref{prop:77} below which follows from
Lemma \ref{lem:710711}. 
\begin{lem}
\label{prop:77} For every $T,\eta>0$, there are $\delta,Q,\xi>0$
such that given any $(v,t)\in\mathcal{G}_{T}(\eta),w_{1}\in W_{\delta}^{s}(v)$
and $w_{2}\in f_{-t}W_{\delta}^{u}(f_{t}v)$, for every $0\leq\tau\leq t$
we have 
\[
|U_{v}^{u}(\tau)-U_{w_{1}}^{u}(\tau)|\leq Qe^{-\xi\tau},
\]
\[
|U_{v}^{u}(\tau)-U_{w_{2}}^{u}(\tau)|\leq Q(e^{-\xi\tau}+e^{-\xi(t-\tau)}).
\]

\end{lem}

\begin{lem}
\label{lem:710711} For every $T,\eta>0$, there are $\delta,Q$ such
that given any $(v,t)\in\mathcal{G}_{T}(\eta),w\in B_{\tau}(v,\delta)$,
for every $0\leq\tau\leq t$ we have 
\[
|U_{v}^{u}(\tau)-U_{w}^{u}(\tau)|\leq Q\exp\left(-\frac{\eta\tau}{T}\right)+Q\int_{0}^{\tau}\exp\left(-\frac{\eta(\tau-s)}{2T}\right)|K(f_{s}v)-K(f_{s}w)|ds.
\]
\end{lem}
\begin{proof}[Proof of Lemma \ref{prop:77}]
We may choose small $\delta$ so that $w_{1},w_{2}\in\mathcal{G}_{T}(\eta/2)$.
We will use $Q$ to denote a uniform constant that is updated as necessary
when the context is clear.

Since $w_{1}\in W_{\delta}^{s}(v)$, the smoothness of $K$ together
with Lemma \ref{lem:312} implies 
\[
|K(f_{\tau}v)-K(f_{\tau}w_{1})|\leq Qd_{K}(f_{\tau}v,f_{\tau}w_{1})\leq Qd^{s}(f_{\tau}v,f_{\tau}w_{1})\leq Q\exp\left(-\frac{\eta\tau}{4T}\right)
\]
for any $\tau\in[0,t]$. Thus by Lemma \ref{lem:710711}, there exists
$Q>0$, 

\begin{alignat*}{1}
|U_{v}^{u}(\tau)-U_{w_{1}}^{u}(\tau)| & \leq Q\exp\left(-\frac{\eta\tau}{T}\right)+Q\int_{0}^{\tau}\exp(-\frac{\eta(\tau-s)}{2T})\exp(-\frac{\eta s}{4T})ds\\
 & \leq Q\exp\left(-\frac{\eta\tau}{T}\right)+Q\exp(-\frac{\eta\tau}{4T})\int_{0}^{\tau}\exp(-\frac{\eta(\tau-s)}{4T})ds\\
 & \leq Q\exp\left(-\frac{\eta\tau}{T}\right)+Q\tau\exp\left(-\frac{\eta\tau}{4T}\right)\\
 & \leq Qe^{-\xi\tau},
\end{alignat*}
once we fix $\xi<\eta/4T$. Hence $|U_{v}^{u}(\tau)-U_{w_{1}}^{u}(\tau)|\leq Qe^{-\xi\tau}$.

For $w_{2}\in f_{-t}W_{\delta}^{u}(f_{t}v)$, we have the following
estimation for $K$: 
\[
|K(f_{\tau}v)-K(f_{\tau}w_{2})|\leq Qd_{K}(f_{\tau}v,f_{\tau}w_{2})\leq Qd^{u}(f_{\tau-t}f_{t}v,f_{\tau-t}f_{t}w_{2})\leq Q\exp\left(-\frac{\eta(t-\tau)}{4T}\right)
\]
for any $\tau\in[0,t]$. We use Lemma \ref{lem:710711} again and
get: 
\begin{eqnarray*}
|U_{v}^{u}(\tau)-U_{w_{2}}^{u}(\tau)| & \leq & Q\exp\left(-\frac{\eta\tau}{T}\right)+Q\int_{0}^{\tau}\exp(-\frac{\eta(\tau-s)}{2T})\exp\left(-\frac{\eta(t-s)}{4T}\right)ds,\\
 & \leq & Q\exp\left(-\frac{\eta\tau}{T}\right)+Q\exp\left(-\frac{\eta(t-\tau)}{4T}\right)\int_{0}^{\tau}\exp\left(-\frac{3\eta(\tau-s)}{4T}\right)ds\\
 & \leq & Q\exp\left(-\frac{\eta\tau}{T}\right)+Q\exp\left(-\frac{\eta(t-\tau)}{4T}\right).
\end{eqnarray*}
This completes the proof.
\end{proof}

\begin{proof}[Proof of Lemma \ref{lem:710711}]
Without loss of generality, we may assume $U_{w}^{u}(0)\geq U_{v}^{u}(0)$
and let $U_{1}$ be the solution of the Riccati equation along $\gamma_{v}$
with $U_{1}(0)=U_{w}^{u}(0)$. We have 
\[
|U_{v}^{u}(\tau)-U_{w}^{u}(\tau)|\leq|U_{v}^{u}(\tau)-U_{1}(\tau)|+|U_{1}(\tau)-U_{w}^{u}(\tau)|.
\]
Since $U_{w}^{u}(0)\geq U_{v}^{u}(0)$ and both $U_{1}$ and $U_{v}^{u}$
are Riccati solutions along $\gamma_{v}$, we have $U_{1}(\tau)\geq U_{v}^{u}(\tau)=k^{u}(f_{\tau}v)$
for all $\tau$. Hence 
\[
(U_{1}-U_{v}^{u})'=-(U_{1}-U_{v}^{u})(U_{1}+U_{v}^{u})\leq-2k^{u}(f_{\tau}v)(U_{1}-U_{v}^{u})\leq-2\lambda(f_{\tau}v)(U_{1}-U_{v}^{u}).
\]
Thus $(U_{1}(\tau)-U_{v}^{u}(\tau))\exp\big(\int_{0}^{\tau}2\lambda(f_{s}v)ds\big)$
is not increasing. By Lemma \ref{lem:312} we have 
\begin{eqnarray*}
0 & \leq & U_{1}(\tau)-U_{v}^{u}(\tau)\leq(U_{w}^{u}(0)-U_{v}^{u}(0))\exp\left(-\int_{0}^{\tau}2\lambda(f_{s}v)ds\right)\\
 & \leq & Q\exp\left(-\frac{1}{T}\int_{0}^{\tau}\lt(f_{s}v)ds\right)\leq Q\exp\left(-\frac{\eta\tau}{T}\right).
\end{eqnarray*}
Now we estimate $|U_{1}(\tau)-U_{w}^{u}(\tau)|$. We may assume $U_{1}(\tau)>U_{w}^{u}(\tau)$
(the other case is similar). Suppose $U_{1}(s_{0})=U_{w}^{u}(s_{0})$
at $s_{0}<\tau$ and $U_{1}(s)>U_{w}^{u}(s)$ for any $s\in(s_{0},t)$.
By taking difference of the corresponding Riccati equations, for any
$s\in(s_{0},t)$, we have: 
\begin{eqnarray*}
(U_{1}-U_{w}^{u})'(s) & = & -(U_{1}(s)-U_{w}^{u}(s))(U_{1}(s)+U_{w}^{u}(s))+K(f_{s}v)-K(f_{s}w)\\
 & \leq & -2k^{u}(f_{s}w)(U_{1}-U_{w}^{u})(s)+|K(f_{s}v)-K(f_{s}w)|.
\end{eqnarray*}
Thus 
\begin{eqnarray*}
 &  & \frac{d}{ds}\left((U_{1}(s)-U_{v}^{u}(s))\exp\left(\int_{s_{0}}^{s}2k^{u}(f_{a}w)da\right)\right),\\
 & = & \exp\left(\int_{s_{0}}^{s}2k^{u}(f_{a}w)da\right)((U_{1}-U_{w}^{u})'(s)+2k^{u}(f_{s}w)(U_{1}-U_{w}^{u})(s)),\\
 & \leq & \exp\left(\int_{s_{0}}^{s}2k^{u}(f_{a}w)da\right)|K(f_{s}v)-K(f_{s}w)|.
\end{eqnarray*}
Together with Lemma \ref{lem:2.9}, we have 
\begin{eqnarray*}
U_{1}(\tau)-U_{v}^{u}(\tau) & \leq & \exp\left(-\int_{s_{0}}^{\tau}2k^{u}(f_{a}w)da\right)\int_{s_{0}}^{\tau}\exp\left(\int_{s_{0}}^{s}2k^{u}(f_{a}w)da\right)|K(f_{s}v)-K(f_{s}w)|ds\\
 & = & \int_{s_{0}}^{\tau}\exp\left(-\int_{s}^{\tau}2k^{u}(f_{a}w)da\right)|K(f_{s}v)-K(f_{s}w)|ds\\
 & \leq & \int_{s_{0}}^{\tau}\exp\left(-\int_{s}^{\tau}2\lambda(f_{a}w)da\right)|K(f_{s}v)-K(f_{s}w)|ds\\
 & \leq Q & \int_{s_{0}}^{\tau}\exp\left(-\frac{1}{T}\int_{s}^{\tau}\lt(f_{a}w)da\right)|K(f_{s}v)-K(f_{s}w)|ds\\
 & \leq Q & \int_{0}^{\tau}\exp\left(-\frac{\eta(\tau-s)}{2T}\right)|K(f_{s}v)-K(f_{s}w)|ds,
\end{eqnarray*}
where the last inequality follows because $w\in\gt(\eta/2)$. 
\end{proof}
Putting together Lemma \ref{lem:Bowen_stable_unstable_all} and Lemma
\ref{prop:77}, we have the following result:
\begin{thm}
\label{thm:geometric-potential-Bowen}The geometric potential $\varphi^{u}$
has the Bowen property with respect to $\gt(\eta)$ for any $T,\eta>0$. 
\end{thm}

\section{Pressure gap and the proof of Theorem A\label{sec:Pressure-gap-and_Thm_A}}

The aim of this section is to prove Theorem A. In order to do that,
we spend most part of this section on related estimates on pressures,
such as $P(\cdot)$, $P(\Sing,\cdot)$, $P_{\exp}^{\perp}(\cdot),$
and relations between them.

We know when the collection ${\cal C}=X\times[0,\infty)$ we can use
the variational principle to understand the topological pressure $P(\cdot)$.
However, when the collection ${\cal C}$ is not the set of all finite
orbits, the variational principle does not hold any more. Nevertheless,
one can still use empirical measures along orbits segments in ${\cal C}$
to ``understand'' $P({\cal C},\cdot)$. To be more precise, we start
from recalling related terms and estimates given in \cite{BCFT2017}. 

Let $X$ be a compact metric space, ${\cal F}$ be a continuous flow,
and $\vp:X\to\R$ be a continuous potential. Given a collection of
finite orbit segments ${\cal C}\subset X\times[0,\infty)$, for $(x,t)\in{\cal C}$
the \textit{empirical measure} $\delta_{(x,t)}$ is defined as, for
any $\psi\in C(X)$, 
\[
\int\psi d{\cal \delta}_{(x,t)}=\frac{1}{t}\int_{0}^{t}\psi(f_{\tau}x)d\tau.
\]

We further write ${\cal M}_{t}({\cal C})$ for the convex linear combinations
of empirical measures of length $t$, that is, 
\[
{\cal M}_{t}({\cal C}):=\{\sum_{i=1}^{k}a_{i}\delta_{(x_{i},t)}:\ a_{i}\geq0,\ \sum a_{i}=1,\ (x_{i},t)\in{\cal C}\}.
\]
 Finally, let ${\cal M}({\cal {\cal C}})$ denote the set of ${\cal F}-$invariant
Borel probability measures which are limits of measures in ${\cal M_{t}}$,
i.e., 
\[
{\cal M}({\cal C}):=\{\lim_{k\to\infty}\mu_{t_{k}}:\ t_{k}\to\infty,\ \mu_{t_{k}}\in{\cal M}_{t_{k}}({\cal C})\}.
\]

Notice that when ${\cal C}$ contains arbitrary long orbit segments,
${\cal M}({\cal C})$ is a nonempty set.

We recall a useful general result from \cite{BCFT2017}:

\begin{prop}
\cite[Proposition 5.1]{BCFT2017} Suppose $\vp$ is a continuous function,
then 
\[
P({\cal C},\vp)\leq\sup_{\mu\in{\cal M}({\cal C})}P_{\mu}(\vp)
\]
 where $P_{\mu}(\vp):=h_{\mu}+\int\vp d\mu$.
\end{prop}

Let us apply above results to our specific setting: $S$ a closed
surface of genus greater than or equal to 2 without focal points,
${\cal F}$ the geodesic flow for $S$, and $\vp:T^{1}S\to\R$ a continuous
potential.

The following lemma establishes that the pressure of the obstruction
to expansivity is strictly less than the entire pressure. It is a
direct consequence of the flat strip theorem.

\begin{prop}
\cite[Proposition 5.4]{BCFT2017} \label{prop:exp<sing}For a continuous
potential $\varphi$, $P_{{\rm exp}}^{\perp}(\vp)\leq P(\Sing,\vp)$. 
\end{prop}

\begin{proof}
It is a straightforward consequence of the flat strip theorem. Since
the flat strip theorem holds for manifolds without focal points (see
Proposition \ref{prop:no_focal_points}), the proof goes verbatim
as in \cite[Proposition 5.4]{BCFT2017}. 
\end{proof}

\begin{prop}
\label{prop:pressue_gap_bad_orbits}Let ${\cal B}_{T}(\eta)$ be the
collection of bad orbit segments defined as in ${\rm Definition}$$\ $\ref{def:decomposition}.
Then there exist $T_{0}$, $\eta_{0}>0$ such that 
\[
P([{\cal B}_{T_{0}}(\eta_{0})],\vp)<P(\varphi).
\]
\end{prop}
\begin{proof}
Let $D$ be the metric compatible with the weak{*} topology on the
space of $\mathcal{F}$-invariant probability measures $\mathcal{M}({\cal F})$.
Fix $\delta<P(\varphi)-P(\text{Sing},\varphi)$ and choose $\varepsilon>0$
such that 
\[
\mu\in\mathcal{M}({\cal F})\text{ with }D(\mu,\mathcal{M}(\text{Sing}))<\varepsilon\implies P_{\mu}(\varphi)-P(\text{Sing})<\delta.
\]
The existence of such $\varepsilon$ is guaranteed by the upper semi-continuity
of the entropy map $\mathcal{M}({\cal F})\ni\mu\mapsto h_{\mu}(f)$
which follows from the geodesic ${\cal F}:T^{1}S\to T^{1}S$ being
entropy-expansive (see Liu-Wang \cite{Liu:2016de}). From Lemma \ref{lem:lamb_T_and_sing}
and \ref{lem:supp_in_sing}, we have 
\[
\mathcal{M}(\text{Sing})=\bigcap\limits _{\substack{\eta>0,T>0}
}\mathcal{M}_{\lt}(\eta),
\]
where $\mathcal{M}_{\lt}(\eta)=\{\mu\in\mathcal{M}({\cal F})\colon\int\lt d\mu\leq\eta\}$.
Hence, we can find $T_{0},\eta_{0}>0$ such that 
\[
D(\mathcal{M}(\text{Sing}),\mathcal{M}_{\lambda_{T_{0}}}(\eta_{0}))<\varepsilon.
\]
In particular, for any $\mu\in{\cal M}_{\lambda_{T_{0}}}(\eta_{0})$,
we have 
\[
P_{\mu}(\varphi)<P(\text{Sing},\varphi)+\delta.
\]

Since it follows from the definition that ${\cal M}([{\cal B}_{T}(\eta))]\subset{\cal M}_{\lambda_{T}}(\eta)$,
we can verify that the pressure gap $P([{\cal B}_{T_{0}}(\eta_{0})],\varphi)<P(\varphi)$
holds for such choice of $\eta_{0}$ and $T_{0}$:

\[
P([{\cal B}_{T_{0}}(\eta_{0})],\varphi)\leq\sup\limits _{\mu\in{\cal M}([{\cal B}_{T_{0}}(\eta_{0})])}P_{\mu}(\vp)\leq\sup\limits _{\mu\in{\cal M}_{\lambda_{T_{0}}}(\eta_{0})}P_{\mu}(\varphi)\leq\delta+P(\text{Sing},\varphi)<P(\varphi).
\]
This proves the proposition.
\end{proof}

\begin{rem}
We remark that the conclusion of Proposition 7.3 remains to hold if
we take $(T_{0},\eta_{1})$ for any $\eta_{1}\in(0,\eta_{0})$.
\end{rem}

Now, we are ready to prove our first main theorem.

\begin{thm*}
[Theorem A] Let $S$ be a surface of genus greater than or equal
to 2 without focal points and ${\cal F}$ be the geodesic flow over
$S$. Let $\vphi:T^{1}S\to\R$ be a H\"{o}lder continuous potential
or $\vp=q\cdot\vp^{u}$ for some $q\in\R$. Suppose $\vp$ verifies
the pressure gap property $P({\rm Sing},\vphi)<P(\vphi)$, then $\vphi$
has a unique equilibrium state $\mu_{\vphi}$.
\end{thm*}

\begin{proof}
This follows from Theorem \ref{thm:C-T_criteria} (Climenhaga-Thompson's
criteria for uniqueness of equilibrium states). 

We first notice that by Proposition \ref{prop:exp<sing}, $\vp$ satisfies
the first assumption in Theorem \ref{thm:C-T_criteria}. We take the
decomposition $({\cal P},{\cal G},{\cal S})=({\cal B}_{T}(\eta),{\cal G}_{T}(\eta),{\cal B}_{T}(\eta))$
given in Definition \ref{def:decomposition}, then by Proposition
\ref{prop: specification}, Theorem \ref{thm:Holder_Bowen}, and Theorem
\ref{thm:geometric-potential-Bowen}, the conditions $({\rm I})$
and $({\rm II})$ of Theorem \ref{thm:C-T_criteria} are verified. 

Lastly, by Proposition \ref{prop:pressue_gap_bad_orbits}, we know
there exists $(T,\eta)=(T_{0},\eta_{0})$ such that the set of bad
orbit segments has strictly less pressure than that of $\vp$, that
is, $P([{\cal B}_{T_{0}}(\eta_{0})],\vp)<P(\varphi)$, which verifies
the condition $({\rm III})$ of Theorem \ref{thm:C-T_criteria}.
\end{proof}

\section{Properties of the equilibrium states and the proof of Theorem B\label{sec:Prop_eq_states_Thm_B}}

In this section, we aim to prove Theorem B.

\begin{thm*}[Theorem B]
 Let $\vphi:T^{1}S\to\R$ be a H\"{o}lder continuous function or
$\vp=q\cdot\vp^{u}$ satisfying $P({\rm Sing},\vphi)<P(\vphi)$. Then
the equilibrium state $\mu_{\vphi}$ is fully supported, $\mu_{\varphi}({\rm Reg})=1$,
Bernoulli, and is the weak$^{*}$ limit of the weighted regular periodic
orbits.
\end{thm*}

\begin{proof}
The proof is separated into following propositions, namely, Proposition
\ref{prop:full_measure}, \ref{prop:weight_average_regular_orbits},
\ref{prop:fully_supported} and \ref{prop:Bernoulli}.
\end{proof}

\subsection{$\mu_{\vp}(\Reg)=1$ and $\mu_{\vp}$ is Bernoulli }

\begin{prop}
\label{prop:full_measure}$\mu_{\vphi}({\rm Reg})=1$. \end{prop}
\begin{proof}
Since $\mu_{\varphi}$ is the unique equilibrium state for $\vphi$,
we have that $\mu_{\varphi}$ is ergodic (cf. \cite{Climenhaga:2016ut}
Proposition 4.19). Because ${\rm Sing}$ is ${\cal F}-$invariant
we have either ${\rm \mu_{\varphi}({\rm Sing)=1}}$ or ${\rm \mu_{\varphi}({\rm Sing)=0}}$.
Suppose $\mu_{\varphi}({\rm Sing})=1$, then 
\[
P({\rm Sing},\varphi)\geq h_{\mu_{\varphi}}({\cal F})+\int\left.\varphi\right|_{{\rm Sing}}\dd\mu_{\varphi}=P(\varphi),
\]
which contradicts to the pressure gap condition. Thus $\mu_{\varphi}({\rm Reg})=1$. 
\end{proof}

\begin{defn}
[Bernoulli] Let $X$ be a compact metric space and ${\cal F}=(f_{t})_{t\in\R}$
be a continuous flow on $X$ . We call a ${\cal F}-$invariant measure
$\mu$ \textit{Bernoulli} if the system $(X,f_{1},\mu)$ is measurably
isomorphic to a Bernoulli shift, where $f_{1}$ is the time-1 map
of the flow ${\cal F}=(f_{t})_{t\in\R}$. 
\end{defn}

To prove $\mu_{\vp}$ is Bernoulli, we use a result in Ledrappier-Lima-Sarig
\cite{Ledrappier:2016cn}. In order to apply their result, we recall
that for $v\in T^{1}S$, the \textit{Lyapunov exponent} at $v$ is
given by 
\[
\chi(v)=\lim_{t\to\pm\infty}\frac{1}{t}\log\left\Vert \left.df_{t}\right|_{E^{u}(v)}\right\Vert 
\]
whenever both limits exist and are equal. For those $v\in T^{1}S$
such that the Lyapunov exponent $\chi(v)$ exists at $v$ are called
\textit{Lyapunov regular} vectors. It is well-known (by Oseledec multiplicative
ergodic theorem) that the set of Lyapunov regular vectors has full
measure for any ${\cal F}$-invariant probability measure. We denote
$\chi^{+}(v)$ the \textit{positive Lyapunov exponent} at $v$, when
the limit exists, from the Oseledec decomposition.\textbf{ }

\begin{rem}
\label{rem:LE=00003D0_on_Sing}For $v\in\Sing,$ notice $f_{t}$ does
not expand along the unstable bundle $E^{u}(v)$; indeed, the unstable
Jacobi field $J_{v}^{u}$ has constant length for $v\in\Sing$. Thus
we have $\left.\chi\right|_{\Sing}=0$. 
\end{rem}

Using following lemmas, we can show that the unique equilibrium state
for $\mu_{\vp}$ is a hyperbolic measure (i.e., $\chi(v)\neq0$ for
$\mu_{\vp}-$a.e. $v\in T^{1}S$, which is equivalent to $\chi(\mu_{\vp}):=\int\chi(v)d\mu_{\vp}\neq0$
from the ergodicity of $\mu_{\vp})$ which allows us to use Ledrappier-Lima-Sarig
\cite{Ledrappier:2016cn} to conclude $\mu_{\vp}$ is Bernoulli.

\begin{lem}
\label{lem:0LE_implies_Sing}Let $\mu$ be a ${\cal F}$-invariant
probability measure. Suppose for $\mu-$a.e. $v\in T^{1}S$ such that
$\chi(v)=0$ then ${\rm supp}(\mu)\subset\Sing$. 
\end{lem}

\begin{proof}
We first recall that for $\xi\in T_{v}T^{1}S$ we have $||J_{\xi}(t)||^{2}\leq||df_{t}\xi||^{2}$.
Let $\mu\in{\cal M}({\cal F})$ and, without loss of generality, we
may assume $v$ is a Lyapunov regular vector for $\xi\in E_{v}^{u}$.
Then, by Lemma \ref{lem:2.9}
\begin{alignat*}{1}
\chi(v) & =\lim_{t\to\infty}\frac{1}{t}\log||df_{t}|_{E_{v}^{u}}||\\
 & \geq\lim_{t\to\infty}\frac{1}{t}\log||J_{\xi}^{u}(t)||\\
 & \geq\lim_{t\to\infty}\frac{1}{t}\log\left(e^{\int_{0}^{t}k^{u}(f_{\tau}v)d\tau}||J_{\xi}^{u}(0)||\right)\\
 & =\lim_{t\to\infty}\frac{1}{t}\int_{0}^{t}k^{u}(f_{\tau}v)d\tau\geq0.
\end{alignat*}
Integrating with respect to $\mu$, the Birkhoff ergodic theorem gives
$\int\chi(v)d\mu\geq\int k^{u}(v)d\mu\geq0$. Therefore, if $\chi(v)=0$
for $\mu-a.e.$ $v\in T^{1}S$, then $k^{u}(v)=0$ for $\mu-a.e.$
$v\in T^{1}S$; hence, $\lambda(v)=0$ for $\mu-a.e.$ $v\in T^{1}S$.
By Lemma \ref{lem:supp_in_sing}, we are done. 
\end{proof}

\begin{rem}
\label{rem:ergodic_Reg>0_hyperbolic}$\ $

\begin{enumerate}[font=\normalfont]

\item\label{rem:chi(v)>=00003D0}The computation in the above lemma
also points out that if $\mu$ is a ${\cal F}-$invariant probability
measure and $v$ is a Lyapunov regular vector with respect to $\mu$,
then $\chi(v)\geq0$.

\item \label{rem:ergodic+Reg>0=00003D>hyperbolic}If, in addition,
$\mu$ is ergodic and $\mu(\Reg)>0$, we have $\mu$ is hyperbolic.
Indeed, otherwise, there exists $A\subset T^{1}S$ such that $\mu(A)>0$
and $\left.\chi\right|_{A}=0$. Then by the ergodicity of $\mu$ we
have that $\mu(A)=1$. Hence, by Lemma \ref{lem:0LE_implies_Sing},
we get ${\rm supp}\mu\subset{\rm Sing}$ which contradicts $\mu(\Reg)>0$. 

\end{enumerate}
\end{rem}

\begin{prop}
\label{prop:Bernoulli}The unique equilibrium state $\mu_{\varphi}$
is Bernoulli.\end{prop}
\begin{proof}
\cite[Proposition 4.19]{Climenhaga:2016ut} shows that the unique
equilibrium state $\mu_{\vp}$ is ergodic, thus by Proposition \ref{prop:full_measure}
and Remark \ref{rem:ergodic_Reg>0_hyperbolic} (\ref{rem:ergodic+Reg>0=00003D>hyperbolic})
we get that $\mu_{\vp}$ is hyperbolic. Therefore, applying results
in \cite{Ledrappier:2016cn}, we have that $\mu_{\vp}$ is Bernoulli. 
\end{proof}

\begin{rem}
Originally Ledrappier-Lima-Sarig \cite{Ledrappier:2016cn} required
that $h_{\mu}({\cal F})>0$; nevertheless, it has been improved in
Lima-Sarig \cite[Theorem 1.3]{Lima:2014tc} that one only needs to
check $\mu$ is hyperbolic.
\end{rem}

\subsection{$\mu_{\vp}$ is fully supported}

In this subsection, unless stated otherwise, we fix the decomposition
$({\cal P},$ ${\cal G}$, ${\cal S})$ to be $({\cal B}_{T_{0}}(\eta_{0}),{\cal G}_{T_{0}}(\eta_{0}),{\cal B}_{T_{0}}(\eta_{0}))$
where $T_{0}$ and $\eta_{0}$ are given in Proposition \ref{prop:pressue_gap_bad_orbits}.
We notice that this decomposition $({\cal B}_{T_{0}}(\eta_{0}),{\cal G}_{T_{0}}(\eta_{0}),{\cal B}_{T_{0}}(\eta_{0}))$
satisfies the the Climenhaga-Thompson criteria for the uniqueness
of equilibrium states (i.e., Theorem \ref{thm:C-T_criteria}).

For any decomposition $({\cal P},$ ${\cal G}$, ${\cal S})$ and
$M>0$, the collection ${\cal G}^{M}$ is defined as 
\[
{\cal G}^{M}:=\{(x,t):\ s(x,t),p(x,t)\leq M\}.
\]

The following lemma shows that if the decomposition $({\cal P},$
${\cal G}$, ${\cal S})$ satisfies Theorem \ref{thm:C-T_criteria},
then ${\cal G}^{M}$ captures much thermodynamic information whenever
$M$ is large enough.

\begin{lem}
\cite[Lemma 6.1]{BCFT2017}\label{lem:6.1} There exists $M,C,\delta>0$
such that for all $t>0$, 
\begin{equation}
\Lambda({\cal G}^{M},\delta,t)>Ce^{tP(\vp)}.\label{eq:g^M_pressure_lower_bound}
\end{equation}
Hence, for large enough $M$, we have $P({\cal G}^{M},\vphi)=P(\vphi)$.
Moreover, the equilibrium state $\mu_{\phi}$ has the lower Gibbs
property on ${\cal G}^{M}$. More precisely, for any $\rho>0$, there
exist $Q,\tau,M>0$ such that for every $(v,t)\in{\cal G}^{M}$ with
$t\geq\tau$, 
\[
\mu_{\vphi}(B_{t}(v,\rho))\geq Qe^{-tP(\vphi)+\int_{0}^{t}\vphi(f_{s}v)\dd s}.
\]
Therefore, for any $(v,t)\in{\cal G}$ with large $t$ we have $\mu_{\vphi}(B(v,\rho))>0.$ 
\end{lem}

\begin{lem}
\cite[Lemma 6.2]{BCFT2017}\label{lem:6.2}Given $\rho,\eta,T>0$,
there exists $\eta_{1}>0$ so that for any $v\in{\rm Reg}_{T}(\eta)$,
$t>0$, there are $s\geq t$ and $w\in B(v,\rho)$ such that $(w,s)\in{\cal G}_{T}(\eta_{1})$.
In particular, we can choose $\eta_{1}\leq\eta_{0}$ where $\eta_{0}$
is given in Proposition \ref{prop:pressue_gap_bad_orbits}\end{lem}
\begin{proof}
The proof follows, mutatis mutandis, the proof of \cite{BCFT2017}
Lemma 6.2. One only needs to replace the \cite{BCFT2017} Corollary
3.11 in their proof by Lemma \ref{lem:39}, and the last assertion
follows because for $0<\eta'\leq\eta''$, we have $\Reg_{T}(\eta'')\subset\Reg_{T}(\eta')$.
\end{proof}

\begin{prop}
\label{prop:fully_supported}The unique equilibrium state $\mu_{\varphi}$
is fully supported. \end{prop}
\begin{proof}
Since ${\rm Reg}$ dense in $T^{1}M$, it is enough to show that for
any $v\in{\rm Reg}$ and $r>0$ we have $\mu_{\vp}(B(v,r))>0$. 

Since $v\in\text{Reg}$, there exists $t_{0}\in\mathbb{R}$ such that
$\lambda(f_{t_{0}}v)>0$. For convenience, let's denote $v'=f_{t_{0}}v$.
By the continuity of $\lambda,$ there exists $\rho>0$ such that
$\left.\lambda\right|_{B(v',2\rho)}>\eta$ for some $\eta>0$, and
we have $v'\in{\rm Reg}_{T}(2\rho\eta)$. We make sure to pick $\rho$
small enough so that $f_{-t_{0}}B(v',2\rho)\subset B(v,r)$. By Lemma
\ref{lem:6.2}, there exists $\eta_{1}>0$ such that there is $w\in B(v',\rho)$
satisfying $(w,t)\in{\cal G}_{T}(\eta_{1})$ for arbitrary large $t$
(depending on $\rho,\eta$).  

Furthermore, the decomposition $({\cal P},{\cal G},{\cal S})=({\cal B}_{T_{0}}(\eta_{1}),{\cal G}_{T_{0}}(\eta_{1}),{\cal \mathbf{{\cal B}}}_{T_{0}}(\eta_{1}))$
verifies Theorem \ref{thm:C-T_criteria}, assuming that we take $\eta_{1}$
smaller than $\eta_{0}$. Thus by Lemma \ref{lem:6.1} we know $\mu_{\vp}$
satisfies the lower Gibbs property, i.e., 
\[
\mu_{\vp}(B(w,\rho))>0.
\]

Now, because $\mu_{\vp}$ is flow invariant, it follows that 
\[
\mu_{\vp}(B(v,r))\geq\mu_{\vp}(B(v',2\rho))\geq\mu_{\vp}(B(w,\rho))>0.
\]

\end{proof}

\subsection{periodic regular orbits are equidistributed relative to $\mu_{\vp}$}

Let us continue the discussion on ergodic properties of the equilibrium
state. Recall that $S$ is a closed surface without focal point with
genus $\geq2$, and $\vp:T^{1}S\to\R$ is a potential satisfying Theorem
A, and $\mu_{\vp}$ the equilibrium state. In what follows, the good
orbit segment collection ${\cal G}$ always refers to ${\cal G}_{T_{0}}(\eta_{0})$
where $T_{0}$, $\eta_{0}$ are given in Proposition \ref{prop:pressue_gap_bad_orbits}.

\begin{lem}
\label{lem:regular_upper_bound}Suppose $\vp:T^{1}S\to\R$ is a potential
satisfying ${\rm Theorem}\ {\rm A}$. For any $\Delta>0$, there exists
$C>0$ such that 
\[
\Lambda_{\Reg,\D}^{*}(\vp,t)\leq Ce^{tP(\vp)}
\]
for all $t>\Delta$ .
\end{lem}

\begin{proof}
\textbf{Claim:} for all $\Delta>0$ and $\delta<{\rm inj}(S)$, ${\rm Per}_{R}(t-\Delta,t]$
is a $(t,\delta)$-separated set. 

pf. If not, suppose $\g_{1},\g_{2}$ are two closed geodesics in ${\rm Per}_{R}(t-\Delta,t]$
such that $d(\g_{1}(s),\g_{2}(s))\leq\delta$ for all $s\in[0,t]$,
then because $\delta<{\rm inj}(S)$ using the exponential map one
can construct a homotopy between $\g_{1}$ and $\g_{2}$. Since $\g_{1},\g_{2}$
are in the same free homotopy class, their lifts $\widetilde{\g}_{1}$,$\widetilde{\g}_{2}$
are bi-asymptotic. Thus by the Flat Strip Theorem (Proposition \ref{prop:no_focal_points})
$\widetilde{\g}_{1}$ and $\widetilde{\g}_{2}$ bound a flat strip,
and hence they are singular. This contradicts to $\g_{1},\g_{2}\in{\rm Per}_{R}(t-\Delta,t]$. 

Notice that for every $\g\in{\rm Per}_{R}(t-\Delta,t]$, let $v_{\g}$
be a vector tangent to $\g$, we have
\[
|\Phi(\g)-\Phi(v_{\g},t)|\leq\Delta||\vp||,\ {\rm and\ thus\ }\Lambda_{\Reg,\D}^{*}(\vp,t)\leq e^{\Delta||\vp||}\Lambda(\vp,\delta,t).
\]

Lastly, by \cite[Lemma 4.11]{Climenhaga:2016ut}, there exists $C>0$
such that for $t>\Delta$ we have 
\[
\Lambda_{\Reg,\D}^{*}(\vp,t)\leq e^{\Delta||\vp||}\Lambda(\vp,\delta,t)<Ce^{tP(\vp)}.
\]

\end{proof}

\begin{lem}
\label{lem:regular_lower_bound} Suppose $\vp:T^{1}S\to\R$ is a potential
satisfying ${\rm Theorem}\ {\rm A}$. There exists $\Delta,C>0$ such
that 
\[
\frac{C}{t}e^{tP(\vp)}\leq\Lambda_{\Reg,\D}^{*}(\vp,t)
\]
for all large $t$ .
\end{lem}

\begin{proof}
By Lemma \ref{lem:6.1}, we know when $M$ is big, there exists $C_{1},\delta_{1}>0$
such that for all $t>0$ 
\[
C_{1}e^{tP(\vp)}\leq\Lambda({\cal G}^{M},\delta_{1},t).
\]
Hence, it suffices to find find $\delta$, $C_{2},\Delta,s>0$ with
$\delta<\delta_{1}$ such that for any $t>\max\{s,\Delta,2M\}$, we
have 
\[
\Lambda({\cal G}^{M},\delta,t)\leq C_{2}(t+s)\Lambda_{\Reg,\D}^{*}(\vp,t+s).
\]

Indeed, the lemmas follows from these inequalities because 
\[
\Lambda_{\Reg,\Delta}^{*}(\vp,t+s)\geq\frac{C_{1}C_{2}^{-1}}{t+s}e^{tP(\vp)}=\frac{C_{1}C_{2}^{-1}e^{-sP(\vp)}}{t+s}e^{(t+s)P(\vp)}.
\]

We start from labeling sizes of Bowen balls relative to different
propositions. In what follows, we fix $T_{0},\eta_{0}>0$ and $M$
large so that Theorem A and Lemma \ref{lem:6.1} hold. Let $\ep_{1}=\ep_{1}(T_{0},\eta_{0})$
be given in Corollary \ref{cor: closing_lemma_regular}. Since $\vp$
verifies the Bowen property on ${\cal G}^{M}$, let $\ep_{2}=\ep_{2}(T_{0},\eta_{0})$
denote the radius of Bowen balls for the Bowen property. Lastly, because
$S$ is compact and $f_{t}$ is uniformly continuous, for any $\ep>0$,
there exists $\delta_{1}=\delta_{1}(\ep)$ such that when $d_{K}(u,w)<\delta_{1}$
we have $d_{K}(f_{\sigma}u,f_{\sigma}w)<\ep$ for any $\sigma\in[-M,M]$,
without loss of generality, we may choose $\ep<\min\{\ep_{1},\ep_{2}\}$. 

The fist step is to associate each $(v,t)\in{\cal G}^{M}$ with a
regular closed orbit whose length is in the interval $[t-t_{1},t+t_{2}]$
for some $t_{1}$ and $t_{2}$. Recall that for each $(v,t)\in{\cal G}^{M}$
there exists $0<s_{0},p_{0}<M$ such that $(f_{p_{0}}v,t-s_{0}-p_{0})=(v',t')\in{\cal G}$. 

We claim that given $\ep>0$ as above and $\delta_{2}=\min\{\ep,\delta_{1}(\ep)\}$,
there exists $s=s(\delta_{2})$ such that for any $(v',t')\in{\cal G}^{M}$
defined as above, there exists a regular vector $w\in B_{t'}(v',\delta_{2})$
with $f_{t'+\tau}(w)=w$ for some $\tau\in[0,s].$

Indeed, the claim is a direct consequence of Corollary \ref{cor: closing_lemma_regular},
because $(v',t')\in{\cal G\subset}{\cal C}_{T_{0}}(\eta_{0}).$ Moreover,
we also have $f_{-p}w\in B_{t}(v,\ep)$ because $w\in B(v',\delta_{2})\subset B(v',\delta_{1})$
and the choice of $\delta_{1}$. Thus, we have the claim.

Moreover, since $\ep<\ep_{2}$ we have 
\begin{alignat*}{1}
\left|\Phi(v,t)-\Phi(w,t'+\tau)\right| & =\left|\int_{0}^{t}\vp(f_{\sigma}v)d\sigma-\int_{0}^{t'+\tau}\vp(f_{\sigma}w)d\sigma\right|\\
 & \leq\left|\int_{0}^{p_{0}}\vp(f_{\sigma}v)d\sigma+\int_{0}^{t'}\vp(f_{\sigma}v')d\sigma+\int_{0}^{s_{0}}\vp(f_{\sigma+t'}v)d\sigma-\int_{0}^{t'+\tau}\vp(f_{\sigma}w)d\sigma\right|\\
 & \leq(2M+\tau)||\vp||+\left|\int_{0}^{t'}(\vp(f_{\sigma}v')-\vp(f_{\sigma}w))d\sigma\right|,\\
 & \leq(2M+\tau)||\vp||+K
\end{alignat*}
where $K$ is the constant given by the Bowen property.

In sum, given $\ep>0$ as above, we can define a map $\Psi:{\cal G}^{M}\ni(v,t)\mapsto(w,t'+\tau)$
where $w$ is tangent to a regular closed orbit $\g_{w}\in{\rm Per}_{R}(t',t'+\tau]\subset{\rm Per}_{R}(t-2M,t+s]$
and $|\Phi(v,t)-\Phi(\g_{w})|\leq(2M+s)||\vp||+K$. 

We notice that $\left.\Psi\right|_{E_{t}}$ is an injection for every
$(t,\delta)$-separated set $E_{t}\subset{\cal G}^{M}$ provided $\delta>3\ep$
(because for every $(v,t)\in E_{t}$, its image $\Psi(v,t)=(w,t'+\tau)$
satisfies $w\in B_{t}(v,\ep)$). Moreover, because $\Psi(E_{t})$
is $(t,\ep)$-separated, each $\g\in{\rm Per}_{R}(t-2M,t+s]$ has
at most $\frac{t+s}{\ep}$ elements of $\Psi(E_{t})$ tangent to it. 

Hence, for $\delta>3\ep$ and for all $(t,\delta)$-separated set
$E_{t}\subset{\cal G}^{M}$ we have 
\[
\sum_{(v,t)\in E_{t}}e^{\Phi(v,t)}\leq\frac{t+s}{\ep}\cdot e^{(2M+s)||\vp||+K}\cdot\sum_{\g\in{\rm Per}_{R}[t-2M,t+s]}e^{\Phi(\g)}.
\]

The lemma now follows with by setting $C_{2}=e^{(2M+s)||\vp||+K}/\ep$
and $\Delta=2M+s$.
\end{proof}

From the above two lemmas, we can conclude:

\begin{prop}
\label{prop:weight_average_regular_orbits}The unique equilibrium
state $\mu_{\varphi}$ obtained in ${\rm Theorem\ A}$ is a weak$^{*}$
limit of the weighted regular periodic orbits. More precisely, there
exists $\D>0$ such that 
\[
\mu_{\vp}=\lim_{T\to\infty}\frac{\sum_{\g\in{\rm Per}_{R}(T-\D,T]}e^{\Phi(\g)}\delta_{\g}}{\Lambda_{\Reg,\Delta}^{*}(\vp,T)}
\]
 \end{prop}
\begin{proof}
It follows immediately from Lemma \ref{lem:regular_upper_bound},
Lemma \ref{lem:regular_lower_bound}, and Proposition \ref{prop:pressure_equal_equidistribution}.
\end{proof}

\section{The proof of Theorem C and examples\label{sec:Thm_C_exampels}}

In this section, we present the proof of Theorem C and also provide
examples satisfying the pressure gap property. The following lemmas
show that the scalar multiple $q\vp^{u}$ of the geometric potential
possesses the pressure gap property provided $q<1$.

\begin{lem}
If $S$ is a closed surface of genus greater than or equal to 2 without
focal points, then $P(q\varphi^{u})>0=P(\Sing,q\varphi^{u})$ for
each $q\in(-\infty,1)$; in particular, $h_{{\rm top}}(\Sing)=0$.\end{lem}
\begin{proof}
It is a classical result proved by Burns \cite[Theorem, p.6]{Burns:1983dw}
that $\mu_{L}({\rm Reg})>0$ where $\mu_{L}$ is the Liouville measure.
Thus by Lemma \ref{lem:0LE_implies_Sing} and Remark \ref{rem:ergodic_Reg>0_hyperbolic}
we get 

\[
0<\chi(\mu_{L}):=\int_{T^{1}S}\chi(v)d\mu_{L}.
\]
This follows because if $\chi(\mu_{L})=0$, then $\chi(v)=0$ for
$\mu_{L}-a.e.$ $v\in T^{1}S$, and hence, by Lemma \ref{lem:0LE_implies_Sing},
we would have ${\rm supp}(\mu_{L})\subset\Sing$ contradicting $\mu_{L}(\Reg)>0$

Therefore, we know 

\[
0<\chi(\mu_{L})=\chi^{+}(\mu_{L}):=\int_{T^{1}S}\chi^{+}(v)d\mu_{L}=-\int_{T^{1}S}\varphi^{u}d\mu_{L},
\]
where the last equality follows from the Birkhoff ergodic theorem. 

Moreover, by Pesin's entropy formula, we have 
\[
h_{\mu_{L}}({\cal F})=\int_{T^{1}S}\chi^{+}(v)d\mu_{L}.
\]

Thus for $q\in(-\infty,1)$ 
\[
P(q\varphi^{u})\geq h_{\mu_{L}}({\cal F})+\int q\varphi^{u}d\mu_{L}=(q-1)\int\varphi^{u}d\mu_{L}>0.
\]

We claim that $P({\rm Sing},q\vp^{u})=0$. Indeed, for any $\mu\in{\cal M}(\Sing)$,
$P_{\mu}(q\vp^{u}):=h_{\mu}({\cal F})+q\int_{T^{1}S}\vp^{u}d\mu=h_{\mu}({\cal F})+q\int_{\Sing}\vp^{u}d\mu=h_{\mu}({\cal F}).$

By Ruelle's inequality we have $h_{\mu}({\cal F})\leq\int\chi^{+}(v)d\mu=0$
(because $\left.\chi\right|_{\Sing}=0$, see Remark \ref{rem:LE=00003D0_on_Sing}).
Therefore, $P(\Sing,q\vp^{u})=\sup_{\mu\in{\cal M}(\Sing)}P_{\mu}(q\vp^{u})=0$. 
\end{proof}

Now, we are ready to prove Theorem C.

\begin{thm*}
[Theorem C]Suppose $S$ is a closed surface of genus greater than
or equal to 2 without focal points, then the geodesic flow has a unique
equilibrium state $\mu_{q}$ for the potential $q\varphi^{u}$ for
$q<1$. This equilibrium state $\mu_{q}$ satisfies $\mu_{q}(\text{Reg})=1$,
is fully supported, Bernoulli, and is the weak{*} limit of weighted
regular periodic orbits. Furthermore, the map $q\mapsto P(q\vphi^{u})$
is $C^{1}$ for $q<1$, and $P(q\vp^{u})=0$ for $q\geq1$ when $\Sing\neq\emptyset$. 
\end{thm*}

\begin{proof}
By the above lemma, Theorem A, and Theorem B, it remains to show $q\mapsto P(q\vp^{u})$
is $C^{1}$ for $q<1$ and $P(q\vp^{u})=0$ for $q\geq1$ when $\Sing\neq\emptyset$.
We first notice that when $\Sing\neq\emptyset$, we have $P(q\vp^{u})\geq0$.
It is because for any invariant measure $\mu$ such that with ${\rm supp}(\mu)\subset\Sing$,
we have 
\[
h_{\mu}({\cal F})+\int_{T^{1}S}\vp^{u}d\mu=h_{\mu}({\cal F})+\int_{\Sing}\vp^{u}d\mu\geq0.
\]

Moreover, the positive Lyapunov exponent $\chi^{+}$ is the Birkhoff
average of $-\vp^{u}$; thus together with Ruelle's inequality we
have for any invariant measure $\nu\in{\cal M}({\cal F})$: 
\[
h_{\nu}({\cal F})\leq\int_{T^{1}S}\chi^{+}(v)d\nu
\]
and for $q\geq1$
\begin{alignat*}{1}
h_{\nu}({\cal F})+\int\vp^{u}d\nu & =\underset{\leq0}{\underbrace{h_{\nu}({\cal F})-\int_{T^{1}S}\chi^{+}(v)d\nu}}\\
 & \geq h_{\nu}({\cal F})-q\int_{T^{1}S}\chi^{+}(v)d\nu\\
 & =h_{\nu}({\cal F})+q\int_{T^{1}S}\vp^{u}d\nu
\end{alignat*}

Therefore, we have for $q\geq1$ 
\[
P(q\vp^{u})=\sup\{h_{\nu}({\cal F})+q\int_{T^{1}S}\vp^{u}d\nu:\ {\cal \nu\in{\cal M}}({\cal F})\}\leq0;
\]
hence we have $P(q\vp^{u})=0$ for $q\geq1$.

Lastly, Liu-Wang \cite{Liu:2016de} proved that the geodesic flow
is entropy expansive for manifolds without conjugates points. So by
Walters \cite{Walters:1992et}, we know that $q\mapsto P(q\vp^{u})$
is $C^{1}$ at where $q\vp^{u}$ has a unique equilibrium state. In
particular, we know $q\mapsto P(q\vp^{u})$ is $C^{1}$ for $q<1$. 
\end{proof}

The proposition below gives us an easy criteria for the pressure gap
property.

\begin{prop}
\cite[Lemma 9.1]{BCFT2017} Let $S$ be a closed surface of genus
greater than or equal to 2 without focal points and $\varphi:T^{1}S\to\mathbb{R}$
continuous. If 
\[
\sup_{v\in{\rm Sing}}\varphi(v)-\inf_{v\in T^{1}S}\varphi(v)<h_{{\rm top}}(\mathcal{F}),
\]
 then $P({\rm Sing},\vp)<P(\varphi)$. In particular, constant functions
have the pressure gap property.\end{prop}
\begin{proof}
The proof follows from the variational principle. More precisely, 

\begin{alignat*}{1}
\sup_{v\in{\rm Sing}}\varphi(v)-\inf_{v\in T^{1}S}\varphi(v) & <h_{{\rm top}}(\mathcal{F})-\underset{=0}{\underbrace{h_{{\rm top}}({\rm Sing})}}\\
\iff\sup_{v\in{\rm Sing}}\varphi(v)+h_{{\rm top}}(\Sing) & <h_{{\rm top}}(\mathcal{F})+\inf_{v\in T^{1}S}\varphi(v)
\end{alignat*}
and 
\[
P(\Sing,\vp)\leq h_{{\rm top}}(\Sing)+\sup_{v\in{\rm Sing}}\varphi(v)<h_{{\rm top}}(\mathcal{F})+\inf_{v\in T^{1}S}\varphi(v)\leq P(\vp).
\]

\end{proof}

By the above proposition, the following class of potentials also possesses
the pressure gap property.

\begin{cor}
Let $S$ be a closed surface of genus greater than or equal to 2 without
focal points and $\varphi:T^{1}S\to\mathbb{R}$ continuous. If $\left.\vp\right|_{\Sing}=0$
and $\vp\geq0$, then $P(\Sing,\vp)<P(\vp)$.
\end{cor}

\bibliographystyle{amsalpha}
\bibliography{/Users/nyima/Dropbox/TEX/Bibtex/BIB}

\end{document}